\documentclass[11pt]{amsart}

\usepackage{amsfonts}
\usepackage{amssymb}
\usepackage{amsthm}
\usepackage{amsmath}
\usepackage{amscd}
\usepackage{hyperref}
\usepackage{enumitem}
\usepackage{tikz}
\usepackage{pgf}
\usepackage{graphicx}
\usepackage[top=3.5cm, bottom=3.5cm, left=2.5cm, right=2.5cm]{geometry}
\usepackage[capitalise]{cleveref}

\newtheorem{prop}{Proposition}[section]
\newtheorem{theorem}[prop]{Theorem}
\newtheorem{lemma}[prop]{Lemma}

\newtheorem{corollary}[prop]{Corollary}

\newtheorem*{theorem*}{Theorem}

\theoremstyle{definition}

\theoremstyle{remark}

\newtheorem*{remark*}{Remark}
\newtheorem*{remarks*}{Remarks}
\newtheorem{remark}[prop]{Remark}

\theoremstyle{theorem}

\newcommand{\N}{\mathbb{N}}

\newcommand{\fd}{{\lfloor d\rfloor}}

\newcommand{\ph}{{\varphi}}

\newcommand{\Sym}{\text{\textup{Sym}}}

\newcommand{\Aut}{\text{\textup{Aut}}\,}
\newcommand{\Cay}{\mathcal{C}}

\newcommand*{\diam}{\mathop{\textup{diam}}\nolimits}

\newcommand*\diff{\mathop{}\!\mathrm{d}}

\renewcommand{\hat}{\widehat}

\numberwithin{equation}{section}

\title[Vertex-transitive graphs of polynomial growth]{A finitary structure theorem for vertex-transitive graphs of polynomial growth}
\date{}
\author{Romain Tessera}
\address{Institut de Math\'ematiques de Jussieu-Paris Rive Gauche, France}
\email{tessera@phare.normalesup.org}
\author{Matthew C. H. Tointon}
\address{School of Mathematics, University of Bristol, United Kingdom}
\email{m.tointon@bristol.ac.uk}
\thanks{M. Tointon was supported by the Stokes Research Fellow at Pembroke College, Cambridge; by grant FN 200021\_163417/1 of the Swiss National Fund for scientific research; and by a travel grant from the Pembroke College Fellows' Research Fund.}

\begin{document}
\maketitle\begin{abstract}
We prove a quantitative, finitary version of Trofimov's result that a connected, locally finite vertex-transitive graph $\Gamma$ of polynomial growth admits a quotient with finite fibres on which the action of $\Aut(\Gamma)$ is virtually nilpotent with finite vertex stabilisers. We also present some applications. We show that a finite, connected vertex-transitive graph $\Gamma$ of large diameter admits a quotient with fibres of small diameter on which the action of $\Aut(\Gamma)$ is virtually abelian with vertex stabilisers of bounded size. We also show that $\Gamma$ has moderate growth in the sense of Diaconis and Saloff-Coste, which is known to imply that the mixing and relaxation times of the lazy random walk on $\Gamma$ are quadratic in the diameter. These results extend results of Breuillard and the second author for finite Cayley graphs of large diameter. Finally, given a connected, locally finite vertex-transitive graph $\Gamma$ exhibiting polynomial growth at a single, sufficiently large scale, we describe its growth at subsequent scales, extending a result of Tao and an earlier result of our own for Cayley graphs. In forthcoming work we will give further applications.
\end{abstract}
\tableofcontents

\section{Introduction}
This paper is part of an ongoing project to answer various structural, geometric and analytic questions about vertex-transitive graphs. A \emph{vertex-transitive graph} $\Gamma$ is a graph whose automorphism group acts transitively on its vertex set. To a certain extent, one can use algebraic methods to study a vertex-transitive graph via its automorphism group. The main aim of the present paper is to approximate certain vertex-transitive graphs by \emph{Cayley graphs}, where such methods are much more directly applicable (we recall the definition of a Cayley graph shortly). We also illustrate the utility of such an approximation by presenting a number of corollaries concerning the structure and geometry of certain vertex-transitive graphs, as well as the behaviour of random walks on them. We will refine some of these and present further applications in our forthcoming work \cite{ttBK,ttLie}.

There are a number of different aspects to our results, so in order to keep this introduction as easy as possible to read we start by presenting slightly simplified versions of them; we state more detailed results in \cref{sec:detail}, where we also fix our notation and terminology (including some standard terms that we use freely in this introduction).

Given a group $G$ with a symmetric generating set $S$, we define the \emph{Cayley graph} $\Cay(G,S)$ to be the graph whose vertices are the elements of $G$ with an edge between $x$ and $y$ precisely when $x\ne y$ and there exists $s\in S$ such that $xs=y$. Note that there is a natural vertex-transitive action of $G$ on $\Cay(G,S)$ given by $g\cdot x=gx$, so Cayley graphs are always vertex transitive. 

For the simplified version of our results we concentrate on a type of geometric approximation called a \emph{quasi-isometry}. Quasi-isometries are certain maps between geometric spaces that preserve asymptotic geometry: roughly, to say that two metric spaces are quasi-isometric is to say that their geometric properties are indistinguishable on a `large scale'. For example, any finite graph will be quasi-isometric to a point, which can be interpreted as saying that `every finite graph appears like a point if you look at it from far enough away'; thus, a star cluster in the night sky may appear to the naked eye like a single bright star.

We now give the formal definition. Given $C\geq 1$ and $K\geq 0$ and metric spaces $X,Y$, a map $f:X\to Y$ is said to be a \emph{$(C,K)$-quasi-isometry} if
\[
C^{-1}d(x,y)-K\leq d(f(x),f(y))\leq Cd(x,y)+K
\]
for every $x,y\in X$, and if every $y\in Y$ lies at distance at most $K$ from $f(X)$. We call $C$ and $K$ the \emph{parameters} of the quasi-isometry.
We occasionally say simply that $f$ is a \emph{quasi-isometry} to mean that there exist some finite, but potentially arbitrarily large, $C$ and $K$ such that $f$ is a $(C,K)$-quasi-isometry. We also occasionally write that $f$ is a $(C,?)$-quasi-isometry to mean that $f$ is a $(C,K)$-quasi-isometry for some finite but potentially arbitrarily large $K$.

Thus, one way of saying that a locally finite vertex-transitive graph $\Gamma$ can be approximated by a locally finite Cayley graph $\Gamma'$ is to say that $\Gamma$ and $\Gamma'$ are quasi-isometric. It is known that such an approximation is not available in the case of an arbitrary vertex-transitive graph: Diestel and Leader \cite{dl} constructed a certain locally finite vertex-transitive graph that they conjectured not to be quasi-isometric to any locally finite Cayley graph, and
Eskin, Fisher and Whyte \cite{efw} confirmed this conjecture, answering a question of Woess \cite[Problem 1]{woess}. Nonetheless, Trofimov \cite{trof} has shown that a certain class of vertex-transitive graphs, namely the vertex-transitive graphs of \emph{polynomial growth}, which we define shortly, are always quasi-isometric to locally finite Cayley graphs.

Before we state Trofimov's result it will be convenient to define two terms. First, given a locally finite vertex-transitive graph $\Gamma$, write $\beta_\Gamma(n)$ for the cardinality of a ball of radius $n$ in $\Gamma$ with respect to the graph metric. The graph $\Gamma$ is said to have \emph{polynomial growth} if there exist constants $C,d\ge0$ such that $\beta_\Gamma(n)\le Cn^d$ for every $n\in\N$. Second, we call a group \emph{virtually nilpotent} if it possesses a nilpotent subgroup of finite index.
\begin{theorem}[Trofimov \cite{trof}]\label{thm:trof.geom}
Let $\Gamma$ be a connected, locally finite vertex-transitive graph of polynomial growth. Then $\Gamma$ is $(1,?)$-quasi-isometric to a locally finite Cayley graph whose underlying group has a nilpotent subgroup of finite index.
\end{theorem}
\cref{thm:trof.geom} is actually a consequence of a more detailed result, which we state below as \cref{thm:trof} (see \cref{rem:trof.QI}).

An important ingredient in the proof of Theorem \ref{thm:trof.geom} is Gromov's famous theorem that if a Cayley graph has polynomial growth then the underlying group is virtually nilpotent \cite{gromov}. This result has been the subject of much work since Gromov's original paper, with Kleiner \cite{kleiner} and Ozawa \cite{ozawa} giving shorter proofs and Shalom and Tao \cite{shalom-tao} giving a more effective result, for example. Of particular relevance to the present paper is the following finitary refinement of Gromov's theorem.
\begin{theorem}[Breuillard--Green--Tao {\cite[Corollary 11.2]{bgt}}]\label{thm:bgt.gromov}
Given $K\ge1$ there exists $n_0=n_0(K)$ such that if $G$ is a group generated by a finite symmetric set $S$ containing the identity, and if $|S^{2n}|\le K|S^n|$ for some $n\ge n_0$, then $G$ is virtually nilpotent.
\end{theorem}
This is indeed a refinement of Gromov's theorem, since, as is well known, polynomial growth implies the existence of some $K\ge1$ and infinitely many $n\in\N$ such that $|S^{2n}|\le K|S^n|$ (for a proof see \cref{lem:poly.pigeon}, below). The main purpose of this paper is to prove an analogous finitary version of \cref{thm:trof.geom}, as follows.
\begin{theorem}[main result, simple form]\label{thm:main.geom}
For every $K\ge1$ there exists $n_0=n_0(K)$ such that the following holds. Let $\Gamma$ be a connected, locally finite vertex-transitive graph, and suppose that there exists $n\ge n_0$ such that $\beta_\Gamma(3n)\le K\beta_\Gamma(n)$. Then $\Gamma$ is $(1,O_K(n))$-quasi-isometric to a locally finite Cayley graph whose underlying group has a nilpotent normal subgroup of rank, step and index at most $O_K(1)$.
\end{theorem}
Again, this is a consequence of a more detailed result, which we state below as \cref{thm:main}.
\begin{remark*}
Since it has no bearing on our applications, we do not consider in this paper whether it would be possible to replace the hypothesis $\beta_\Gamma(3n)\le K\beta_\Gamma(n)$ in the statement of \cref{thm:main.geom} with $\beta_\Gamma(2n)\le K\beta_\Gamma(n)$ as in \cref{thm:bgt.gromov}. If $G$ had a bi-invariant Haar measure when viewed as a topological group as in \cref{sec:top} then this would be possible using results of \cite{tao.product.set}.
\end{remark*}

To give a flavour of the kinds of results we can prove using \cref{thm:main.geom}, we now present two corollaries. We present further corollaries in \cref{sec:detail}, after the more detailed version of \cref{thm:main.geom}.

Breuillard, Green and Tao showed that if $\Gamma$ is a Cayley graph satisfying $\beta_\Gamma(n)\le n^d\beta_\Gamma(1)$
for large enough $n$ then $\beta_\Gamma(m)\le m^{O_d(1)}\beta_\Gamma(1)$ for every $m\ge n$ \cite[Corollary 11.9]{bgt}; thus, if a ball of large enough radius in a Cayley graph exhibits polynomial growth then all larger balls also exhibit polynomial growth. Tao subsequently gave a much more detailed description of the further growth of $\beta_\Gamma(m)$ in this setting. In order to state his result precisely we define a function $f:[1,\infty)\to[1,\infty)$ to be \emph{piecewise monomial} if there exist $1=x_0<x_1<\ldots<x_k=\infty$ and $C_1,\ldots,C_k$ and $d_1,\ldots,d_k\ge0$ such that $f(x)=C_ix^{d_i}$ whenever $x\in[x_{i-1},x_i)$. We call the restrictions $f|_{[x_{i-1},x_i)}$ the \emph{pieces} of $f$, and each $d_i$ the \emph{degree} of the piece $f|_{[x_{i-1},x_i)}$. Tao's result states that if $\Gamma$ is a Cayley graph satisfying $\beta_\Gamma(n)\le n^d\beta_\Gamma(1)$ for large enough $n$ then there exists a non-decreasing continuous piecewise-monomial function $f:[1,\infty)\to[1,\infty)$ with $f(1)=1$ and at most $O_d(1)$ distinct pieces, each of degree a non-negative integer at most $O_d(1)$, such that $\beta_\Gamma(mn)\asymp_df(m)\beta_\Gamma(n)$ for every $m\in\N$ \cite[Theorem 1.9]{tao.growth} (for an explanation of the notation $\asymp$ and other asymptotic notation see the `Notation' subsection, below). Using \cref{thm:main}, we can extend Tao's result to arbitrary connected locally finite vertex-transitive graphs, as follows.
\begin{corollary}\label{cor:tao.growth}
Given $d>0$ there exists $n_0=n_0(d)$ such that if $n\ge n_0$, and if $\Gamma$ is a connected locally finite vertex-transitive graph such that $\beta_\Gamma(n)\le n^d\beta_\Gamma(1)$,
then there exists a non-decreasing continuous piecewise-monomial function $f:[1,\infty)\to[1,\infty)$ with $f(1)=1$ and at most $O_d(1)$ distinct pieces, each of degree a non-negative integer at most $O_d(1)$, such that
$\beta_\Gamma(mn)\asymp_df(m)\beta_\Gamma(n)$
for every $m\in\N$.
\end{corollary}

In our forthcoming paper \cite{ttLie} we will give an explicit, sharp bound on the maximum degree of the function $f$ in \cref{cor:tao.growth}. We will also give an explicit polynomial bound on the number of monomial pieces that make up $f$.

It is not hard to check that \cref{cor:tao.growth} extends, to vertex-transitive graphs, the Breuillard--Green--Tao result that if $\beta_\Gamma(n)\le n^d\beta_\Gamma(1)$ for some $n\ge n_0$ then $\beta_\Gamma(m)\le m^{O_d(1)}\beta_\Gamma(1)$ for every $m\ge n$. Tao gave an example to show that the bound $O_d(1)$ in this inequality cannot in general be taken to be $d$ \cite[Example 1.11]{tao.growth}. Nonetheless, in the first paper in this series \cite{tt} we showed that if one assumes the more stringent condition $\beta_\Gamma(n)\le n^d$ on a Cayley graph $\Gamma$ in place of $\beta_\Gamma(n)\le n^d\beta_\Gamma(1)$ then one can replace the bound $O_d(1)$ with a bound of $d$. This answered a question of Benjamini. As another corollary of \cref{thm:main}, we extend this result to vertex-transitive graphs, as follows.
\begin{corollary}\label{cor:persist.abs}
Given $d\in\N$ there exists $C=C_d>0$ such that if $\Gamma$ is a locally finite vertex-transitive graph satisfying
\begin{equation}\label{eq:persist.abs.hyp}
\beta_\Gamma(n)\le\frac{n^{d+1}}{C}
\end{equation}
for some $n\ge C$ then for every $m\ge n$ we have
\[
\beta_\Gamma(m)\le C\left(\frac{m}{n}\right)^d\beta_\Gamma(n).
\]
Moreover, if $\beta_\Gamma(n)\le n$ then given $x\in\Gamma$ we have $\Gamma=B_\Gamma(x,n)$.
\end{corollary}
In another paper \cite{ttBK} we use \cref{cor:persist.abs} to prove a number of results about \emph{escape probabilities} of random walks in vertex-transitive graphs, in particular verifying two conjectures of Benjamini and Kozma \cite[Conjectures 4.1 \& 4.2]{bk}.

\begin{remark*}
The arguments of this paper are completely explicit and effective. However, we rely on two ineffective results from the literature on approximate groups, so we are unable to make the bounds in our results explicit at present. To be precise, the only ineffective quantities in Theorem \ref{thm:main.geom} are $n_0$ and the index of the nilpotent subgroup. The fact that the other quantities are effective plays a crucial role in a forthcoming paper on Gromov--Hausdorff approximation of balls of polynomial size in vertex-transitive graphs \cite{ttLie}.
The ineffective results we rely on are Breuillard, Green and Tao's celebrated classification finite approximate groups \cite{bgt}, which is the main ingredient in \cref{thm:bgt.gromov}, and Carolino's corresponding result for relatively compact approximate subgroups of locally compact groups, which we state below as \cref{thm:carolino}. The source of ineffectiveness is essentially the same in each of these two results, and if that could be eliminated then our results would in principle all be effective. \end{remark*}

\subsection*{Notation}Throughout this paper we make heavy use of the convention that if $X,Y$ are real quantities  and $z_1,\ldots,z_k$ are variables or constants then the expressions $X\ll_{z_1,\ldots,z_k}Y$ and $Y\gg_{z_1,\ldots,z_k}X$ each mean that there exists a constant $C>0$ depending only on $z_1,\ldots,z_k$ such that $X$ is always at most $CY$. Moreover, the notation $O_{z_1,\ldots,z_k}(Y)$ denotes a quantity that is at most a certain constant (depending on $z_1,\ldots,z_k$) multiple of $Y$, while $\Omega_{z_1,\ldots,z_k}(X)$ denotes a quantity that is at least a certain positive constant (depending on $z_1,\ldots,z_k$) multiple of $X$. Thus, for example, the meaning of the notation $X\le O(Y)$ is identical to the meaning of the notation $X\ll Y$. We also write $X\asymp_{z_1,\ldots,z_k}Y$ to mean that $X\ll_{z_1,\ldots,z_k}Y$ and $Y\ll_{z_1,\ldots,z_k}X$ both hold.

We use the notation $H<G$ to mean that $H$ is a subgroup of $G$, with equality permitted.

\subsection*{Acknowledgements} We thank Itai Benjamini, Emmanuel Breuillard, Persi Diaconis, Jonathan Hermon and Ariel Yadin for helpful conversations, and three anonymous referees for a number of comments on and corrections to an earlier version of this paper.

\section{Detailed statement of results and further corollaries}\label{sec:detail}
In this section we state in full our main results and some further corollaries. We also establish some notation and present the necessary definitions.

Throughout this paper, we use the term \emph{graph} to mean an undirected graph without loops or multiple edges. Given a graph $\Gamma$, we abuse notation slightly by also writing $\Gamma$ for the vertex set of $\Gamma$. For example, if $\Gamma$ is finite then $|\Gamma|$ is the number of vertices, and $x\in\Gamma$ means that $x$ is a vertex of $\Gamma$. Given $x,y\in\Gamma$ we write $x\sim y$ to mean that there is an edge between $x$ and $y$. We write $d_\Gamma(x,y)$ for the number of edges in the shortest path between $x$ and $y$, often dropping the subscript $\Gamma$ when it is clear from context. Note that this makes $\Gamma$ into a metric space; we call $d_\Gamma$ the \emph{graph metric} on $\Gamma$. Given $x\in\Gamma$, we write $\deg(x)$ for the degree of $x$. Given in addition $n\in\N$, we write $B_\Gamma(x,n)$ for the ball of radius $n$ centred at $x$; thus $B_\Gamma(x,n)=\{y\in\Gamma:d(x,y)\le n\}$.

We write $\Aut(\Gamma)$ for the group of automorphisms of the graph $\Gamma$, acting from the left. We say that a subgroup $G<\Aut(\Gamma)$ is \emph{transitive} if for every $x,y\in\Gamma$ there exists $g\in\Aut(\Gamma)$ such that $g(x)=y$. We say that $\Gamma$ is \emph{vertex transitive} if $\Aut(\Gamma)$ is transitive. Note that automorphisms of $\Gamma$ are the same as isometries of $\Gamma$ with respect to the metric $d_\Gamma$, and on certain occasions it will be more natural to refer to them as such.

In a vertex-transitive graph $\Gamma$, the degree of every vertex is the same; we write $\deg(\Gamma)$ for this degree. The cardinality $|B_\Gamma(x,n)|$ of the ball of radius $n$ centred at $x$ is also the same for every $x$; we write $\beta_\Gamma(n)$ for this cardinality.

When defining a Cayley graph $\Cay(G,S)$ we do not insist that $S$ be finite, but we will generally assume that $S$ contains the identity element $1$; this is largely for notational convenience, meaning, for example, that $B_{\Cay(G,S)}(1,n)=S^n$. We often abbreviate the distance $d_{\Cay(G,S)}$ as simply $d_S$. If $H\lhd G$ then, writing $\pi:G\to G/H$ for the quotient homomorphism, we also often abbreviate $\Cay(G/H,\pi(S))$ as $\Cay(G/H,S)$.

If $\Gamma$ is a vertex-transitive graph and $H<\Aut(\Gamma)$ is a subgroup then we define $\Gamma/H$ to be the quotient graph with vertices $\{H(x):x\in\Gamma\}$, and $H(x)\sim H(y)$ in $\Gamma/H$ if and only if there exists $x_0\in H(x)$ and $y_0\in H(y)$ such that $x_0\sim y_0$ in $\Gamma$. Note that $\Gamma/H$ is trivial if and only if $H$ is transitive. We call the sets $H(x)\subset\Gamma$ with $x\in\Gamma$ the \emph{fibres} of the projection $\Gamma\to\Gamma/H$.

If $G$ is another subgroup of $\Aut(\Gamma)$, we say that the quotient graph $\Gamma/H$ is \emph{invariant under the action of $G$ on $\Gamma$} if for every $g\in G$ and $x\in\Gamma$ there exists $y\in\Gamma$ such that $gH(x)=H(y)$. We will see in \cref{lem:G/H',lem:H.inv.for.G} that if $H$ is normalised by $G$ then $\Gamma/H$ is invariant under the action of $G$, and the action of $G$ on $\Gamma$ descends to an action of $G$ on $\Gamma/H$. We write $G_{\Gamma/H}$ for the image of $G$ in $\Aut(\Gamma/H)$ induced by this action; thus $G_{\Gamma/H}$ is the quotient of $G$ by the normal subgroup $\{g\in G:gH(x)=H(x)\text{ for every }x\in\Gamma\}$.

The following version of Trofimov's theorem essentially arises from an argument given by Woess.
\begin{theorem}[Trofimov \cite{trof}; Woess \cite{woess}]\label{thm:trof}
Let $\Gamma$ be a connected, locally finite vertex-transitive graph of polynomial growth, and let $e\in\Gamma$. Let $G<\Aut(\Gamma)$ be a transitive subgroup. Then there is a normal subgroup $H\lhd G$ such that
\begin{enumerate}[label=(\roman*)]
\item every fibre of the projection $\Gamma\to\Gamma/H$ is finite;
\item\label{item:woess.G/H} $G_{\Gamma/H}=G/H$;
\item $G_{\Gamma/H}$ is virtually nilpotent;
\item the set $S=\{g\in G_{\Gamma/H}:d_{\Gamma/H}(g(H(e)),H(e))\le1\}$ is a finite symmetric generating set for $G_{\Gamma/H}$; and
\item every vertex stabiliser of the action of $G_{\Gamma/H}$ on $\Gamma/H$ is finite;
\end{enumerate}
\end{theorem}
\begin{remark}\label{rem:trof.QI}
It follows from \cref{lem:quotient.QI,lem:aut.QI}, below, that in the setting of \cref{thm:trof} there exist maps
\begin{equation}\label{eq:QI.initial}
\begin{CD}
@.\Gamma \\
@.@VV(1,?)V    \\
\Cay(G_{\Gamma/H},S) @>(1,1)>> \Gamma/H\\
\end{CD}
\end{equation}
that are quasi-isometries with parameters as indicated. We will see in \cref{lem:QI.elem} that the quasi-isometries represented on the diagram \eqref{eq:QI.initial} can be inverted and composed to produce the $(1,?)$-quasi-isometry $\Gamma\to\Cay(G_{\Gamma/H},S)$ of \cref{thm:trof.geom}. 
\end{remark}

Our main result is the following refinement of \cref{thm:trof}.

\begin{theorem}[main result, detailed form]\label{thm:main}
For every $K\ge1$ there exists $n_0=n_0(K)$ such that the following holds. Let $\Gamma$ be a connected, locally finite vertex-transitive graph with a distinguished vertex $e$, and suppose that there exists $n\ge n_0$ such that $\beta_\Gamma(3n)\le K\beta_\Gamma(n)$.
Let $G<\Aut(\Gamma)$ be a transitive subgroup. Then there is a normal subgroup $H\lhd G$ such that
\begin{enumerate}[label=(\roman*)]
\item\label{item:i} every fibre of the projection $\Gamma\to\Gamma/H$ has diameter at most $O_K(n)$;
\item\label{item:ii} $G_{\Gamma/H}=G/H$;
\item\label{item:iii} $G_{\Gamma/H}$ has a nilpotent normal subgroup of rank, step and index at most $O_K(1)$;
\item\label{item:iv} the set $S=\{g\in G_{\Gamma/H}:d_{\Gamma/H}(g(H(e)),H(e))\le1\}$ is a finite symmetric generating set for $G_{\Gamma/H}$;
\item\label{item:v} every vertex stabiliser of the action of $G_{\Gamma/H}$ on $\Gamma/H$ has cardinality $O_K(1)$; and
\item\label{item:vi} there is a $(1,O_K(n))$-quasi-isometry $\Gamma\to\Cay(G_{\Gamma/H},S)$.
\end{enumerate}
\end{theorem}
The \emph{diameter} of a subset $A$ of a metric space is given by $\sup_{x,y\in A}d(x,y)$.
\begin{remarks*}\mbox{}
\begin{enumerate}[label=(\arabic*)]
\item It is not entirely surprising that there should be a finitary refinement of \cref{thm:trof} in the direction of \cref{thm:main}, given some advances that have been made since Woess's paper. Indeed, in \cref{thm:prelim} we obtain a partial version of \cref{thm:main} essentally by replacing two ingredients in Woess's proof with more advanced results that were not available to Woess at the time. This alone is not good enough for our applications, however. An important deficiency of \cref{thm:prelim} compared to \cref{thm:main} is that \cref{thm:prelim} does not give a bound on the size of the vertex stabilisers of the action of $G_{\Gamma/H}$ on $\Gamma/H$, a conclusion that is crucial for our proof of \cref{cor:persist.abs}, below, for example.
\item One may check that if the vertex stabilisers of the action of a transitive group $G$ on a graph $\Gamma$ are trivial then $\Gamma$ is isomorphic to the Cayley graph of $G$ with respect to the generating set $\{g\in G:d(g,e)\le1\}$. Conclusion \ref{item:v} of \cref{thm:main} can therefore be interpreted as saying that $\Gamma/H$ is `close to' $\Cay(G_{\Gamma/H},S)$.
\end{enumerate}
\end{remarks*}

\bigskip\noindent We now present some applications of \cref{thm:main} to vertex-transitive graphs of polynomial growth. We start with a corollary confirming that \cref{thm:main} is a refinement of \cref{thm:trof}.

\begin{corollary}\label{cor:trof}
For all $d\ge0$ and $\lambda\in(0,1)$ there exists $n_0=n_0(d,\lambda)\ge1$ such that the following holds. Let $\Gamma$ be a connected, locally finite vertex-transitive graph with a distinguished vertex $e$, and suppose that there exists $n\ge n_0$ such that
\begin{equation}\label{eq:trof.one.scale}
\beta_\Gamma(n)\le n^d\beta_\Gamma(1).
\end{equation}
Let $G<\Aut(\Gamma)$ be a transitive subgroup. Then there is a normal subgroup $H\lhd G$ such that
\begin{enumerate}[label=(\roman*)]
\item\label{item:c.i} every fibre of the projection $\Gamma\to\Gamma/H$ has diameter at most $n^\lambda$;
\item\label{item:c.ii} $G_{\Gamma/H}=G/H$;
\item\label{item:c.iii} $G_{\Gamma/H}$ has a nilpotent normal subgroup of rank, step and index at most $O_{d,\lambda}(1)$;
\item\label{item:c.iv} the set $S=\{g\in G_{\Gamma/H}:d_{\Gamma/H}(g(H(e)),H(e))\le1\}$ is a finite symmetric generating set for $G_{\Gamma/H}$;
\item\label{item:c.v} every vertex stabiliser of the action of $G_{\Gamma/H}$ on $\Gamma/H$ has cardinality $O_{d,\lambda}(1)$; and
\item\label{item:c.vi} there is a $(1,n^{\lambda})$-quasi-isometry $\Gamma\to\Cay(G_{\Gamma/H},S)$.
\end{enumerate}
\end{corollary}

In the case that $\Gamma$ is finite, a natural value of $n$ at which to seek to apply \cref{cor:trof} is the diameter of $\Gamma$, which we write $\diam(\Gamma)$. Note that $\diam(\Gamma)=\min\{n:B_\Gamma(e,n)=\Gamma\}$. For $n=\diam(\Gamma)$, the hypothesis \eqref{eq:trof.one.scale} translates to the large-diameter condition
\begin{equation}\label{eq:large.diam}
\diam(\Gamma)\ge\left(\frac{|\Gamma|}{\beta_\Gamma(1)}\right)^\delta,
\end{equation}
with $\delta=\frac{1}{d}$. Breuillard and the second author \cite{bt} have studied Cayley graphs satisfying \eqref{eq:large.diam}, calling such Cayley graphs \emph{$\delta$-almost flat}. For brevity, when writing the diameter of a Cayley graph $\Cay(G,S)$ we write $\diam_S(G)$ in place of $\diam(\Cay(G,S))$.

\cref{thm:bgt.gromov} shows that an almost-flat Cayley graph possesses a large virtually nilpotent quotient. One of Breuillard and the second author's results shows that one can refine this further to a large virtually abelian quotient \cite[Theorem 4.1 (2)]{bt}. They also show that the finite-index abelian subgroup of this quotient in turn admits a cyclic quotient of diameter comparable to that of the original Cayley graph \cite[Theorem 4.1 (3)]{bt}. Our next corollaries extend these results to vertex-transitive graphs. 

\begin{corollary}\label{cor:trof.finite}
For every $\delta>0$ and $\lambda>0$ there exists $n_0=n_0(\delta,\lambda)\ge1$ such that the following holds. Let $\Gamma$ be a connected, finite vertex-transitive graph with a distinguished vertex $e$, and suppose that $\diam(\Gamma)\ge n_0$ and
\begin{equation}\label{eq:trof.finite.hyp}
\diam(\Gamma)\ge\left(\frac{|\Gamma|}{\beta_\Gamma(1)}\right)^\delta.
\end{equation}
Let $G<\Aut(\Gamma)$ be a transitive subgroup. Then there is a normal subgroup $H\lhd G$ such that
\begin{enumerate}[label=(\roman*)]
\item\label{item:c.i.fin} every fibre of the projection $\Gamma\to\Gamma/H$ has diameter at most $\diam(\Gamma)^{\frac{1}{2}+\lambda}$;
\item\label{item:c.ii.fin} $G_{\Gamma/H}=G/H$;
\item\label{item:c.iii.fin} $G_{\Gamma/H}$ has an abelian subgroup of rank and index at most $O_{\delta,\lambda}(1)$;
\item\label{item:c.iv.fin} the set $S=\{g\in G_{\Gamma/H}:d_{\Gamma/H}(g(H(e)),H(e))\le1\}$ is a symmetric generating set for $G_{\Gamma/H}$;
\item\label{item:c.v.fin} every vertex stabiliser of the action of $G_{\Gamma/H}$ on $\Gamma/H$ has cardinality $O_\delta(1)$; and
\item\label{item:c.vi.fin} there is a $(1,\diam(\Gamma)^{\frac{1}{2}+\lambda})$-quasi-isometry $\Gamma\to\Cay(G_{\Gamma/H},S)$.
\end{enumerate}
\end{corollary}

In the next corollary, if $\Gamma$ is a vertex-transitive graph, $H<\Aut(\Gamma)$ is a subgroup, and $X\subset\Gamma$ is a union of fibres of the projection $\Gamma\to\Gamma/H$ then we define $X/H$ to be the quotient of $X$ (with the subspace metric induced from $\Gamma$) by the equivalence relation with equivalence classes $H(x)$ for $x\in X$. Explicitly, $X/H$ is the metric space with points $\{H(x):x\in X\}$ and distance $d_{X/\Gamma}$ defined via
\begin{equation}\label{eq:quot.metric}
d_{X/H}(H(x),H(y))=\min\left\{d_\Gamma(x_1,y_1)+\cdots+d_\Gamma(x_m,y_m):\begin{array}{c}m\in\N\\x_1=x\\y_m=y
\\x_i,y_i\in X\text{ for every }i\\x_{i+1}\in H(y_i)\text{ for every }i\end{array}\right\}.
\end{equation}
Note that if $X=\Gamma$ then this metric agrees with the graph metric on $\Gamma/H$ defined earlier.

\begin{corollary}\label{cor:trof.finite.cyc}
For every $\delta>0$ there exists $n_0=n_0(\delta)\ge1$ such that if $\Gamma$ is a connected, finite vertex-transitive graph with a distinguished vertex $e$ satisfying $\diam(\Gamma)\ge n_0$ and \eqref{eq:trof.finite.hyp}, and if $G<\Aut(\Gamma)$ is a transitive subgroup, then there exist subgroups $U\lhd G'\lhd G$ and a finite symmetric generating set $S'$ for $G'/U$ containing the identity such that
\begin{enumerate}[label=(\roman*)]
\item \label{item:fin.ind}$G'$ has index $O_\delta(1)$ in $G$;
\item \label{item:cyclic}$G'/U$ is cyclic;
\item \label{item:diam}$\diam_{S'}(G'/U)\gg_\delta\diam(\Gamma)$;
\item \label{item:small.stabs}every vertex stabiliser of the action of $G'/U$ on $G'(e)/U$ has size $O_\delta(1)$; and
\item \label{item:biLip} the map $\ph:\Cay(G'/U,S')\to G'(e)/U$ defined by $\ph(gU)=U(g(e))$ is an $(O_\delta(1),O_\delta(1))$-quasi-isometry.
\end{enumerate}
\end{corollary}
Conclusion \ref{item:fin.ind} of \cref{cor:trof.finite.cyc} (in conjunction with \cref{lem:cay-ab.gen.set,lem:cosets.increasing}, below) implies that $\Gamma$ is $(1,O_\delta(1))$-quasi-isometric to $G'(e)$, whilst conclusion \ref{item:biLip} shows that $G'(e)/U$ is $(O_\delta(1),O_\delta(1))$-quasi-isometric to $\Cay(G/U,S')$. In light of conclusions \ref{item:cyclic} and \ref{item:diam}, \cref{cor:trof.finite.cyc} can therefore be thought of as saying that a large diameter of a finite vertex-transitive graph is always explained, ``modulo quasi-isometries'', by the presence of metric-space quotient that is a Cayley graph of a cyclic group with large diameter.

In our forthcoming work \cite{ttLie} we will obtain sharp bounds on the ranks of the nilpotent and abelian groups appearing in \cref{cor:trof} \ref{item:c.iii} and \cref{cor:trof.finite} \ref{item:c.iii.fin}, respectively.

\cref{cor:trof,cor:trof.finite} can be seen as finitary versions of results of Benjamini, Finucane and the first author \cite{bft} describing the \emph{scaling limits} of vertex-transitive graphs of polynomial growth. Briefly, a sequence of compact metric spaces $X_1,X_2,\ldots$ is said to \emph{Gromov--Hausdorff converge} or \emph{GH converge} to a compact metric space $X$ if there exist $(1+o(1),o(1))$-quasi-isometries $X_n\to X$.
The \emph{scaling limit} of the sequence $(X_n)$, if it exists, is the GH limit of the sequence $X_1',X_2',\ldots$, in which each $X_n'$ is the space $X_n$ with the metric scaled by $\diam(X_n)^{-1}$. Benjamini, Finucane and the first author's main result shows that if $(\Gamma_n)$ is a sequence of finite connected vertex-transitive graphs with diameters tending to infinity satisfying $\diam(\Gamma_n)\ge|\Gamma_n|^\delta$ then $(\Gamma_n)$ has a subsequence with a scaling limit that is a torus of finite dimension with an invariant Finsler metric \cite[Theorem 1]{bft}. An analogous result for infinite vertex-transitive graphs satisfying a condition along the lines of \eqref{eq:trof.one.scale} is given by \cite[Theorem 3.2.2]{bft}. See \cite{bft} for precise statements and definitions.

\begin{remark} \cref{cor:trof} implies a strengthening of \cite[Theorem 3.2.2]{bft}. That theorem includes the hypothesis that the graphs $\Gamma_n$ admit discrete transitive groups of automorphisms (see \cref{sec:top} for the meaning of \emph{discrete} in this context). \cref{cor:trof} (combined with \cref{prop:red.to.cayley}, below) shows that each $\Gamma$ is $(1,o(m_n))$-quasi-isometric to a Cayley graph of comparable growth. Since a Cayley graph always admits the underlying group as a discrete transitive group of automorphisms, this shows that the discreteness assumption of \cite[Theorem 3.2.2]{bft} can be omitted. This verifies \cite[Conjecture 5.1.1]{bft}. We give a detailed proof of this, with the addition of sharp bounds on the dimension of the limit, in our forthcoming work \cite{ttLie}.
\end{remark}

Breuillard and the second author showed that if a finite Cayley graph has large diameter in the sense of \eqref{eq:large.diam} then it has a property called \emph{moderate growth} \cite[Corollary 1.9]{bt}. This property was introduced by Diaconis and Saloff-Coste \cite{dsc}, and one of the main points of their work was to show that a if $\Gamma$ is Cayley graph of moderate growth then the relaxation and mixing times of the simple random walk on $\Gamma$ are quadratic in $\diam(\Gamma)$. Given constants $A,d$ we define a finite connected vertex-transitive graph $\Gamma$ to have \emph{$(A,d)$-moderate growth} if
\[
\beta_\Gamma(n)\ge\frac{1}{A}\left(\frac{n}{\diam(\Gamma)}\right)^d|\Gamma|
\]
whenever $1\le n\le\diam(\Gamma)$. We refer the reader to \cite{lev-per} for the other relevant definitions, and to \cite{bt,dsc} for details of Diaconis and Saloff-Coste's results.

It is essentially immediate from the definition of moderate growth that a finite connected vertex-transitive graph of moderate growth has large diameter in the sense of \eqref{eq:large.diam}; the result of Breuillard and the second author is a converse to this. Using \cref{thm:main} we prove a similar result for general vertex-transitive graphs, as follows.
\begin{corollary}[equivalence of large diameter and moderate growth]\label{cor:mod.growth}
Let $\Gamma$ be a finite connected vertex-transitive graph. If $\Gamma$ has $(A,d)$-moderate growth then
\begin{equation}\label{eq:mod.growth.conc}
\diam(\Gamma)\ge\frac{1}{A^{1/d}}\left(\frac{|\Gamma|}{\beta_\Gamma(1)}\right)^{1/d}.
\end{equation}
Conversely, for every $\delta\ge0$ there exists $n_0=n_0(\delta)$ such that if $\diam(\Gamma)\ge n_0$ and
\begin{equation}\label{eq:mod.growth.hyp}
\diam(\Gamma)\ge\left(\frac{|\Gamma|}{\beta_\Gamma(1)}\right)^\delta
\end{equation}
then $\Gamma$ has $(O_\delta(1),O_\delta(1))$-moderate growth.
\end{corollary}
\begin{remark}It turns out that some similar results to those of Diaconis and Saloff-Coste hold in vertex-transitive graphs of moderate growth; see \cite[Proposition 8.1]{her-pym}, for example. Thus, \cref{cor:mod.growth} implies that if $\Gamma$ is a finite connected vertex-transitive graph satisfying \eqref{eq:mod.growth.hyp} then the mixing time of the simple random walk on $\Gamma$ is comparable to $\diam(\Gamma)^2$. We will formulate this result more precisely in future work.
\end{remark}

\section{Quotients of vertex-transitive graphs}
In this section we present some of the fundamentals of quotients of vertex-transitive graphs, as well as some results that are more specific to our arguments. Throughout this section and the rest of the paper, given a vertex-transitive graph $\Gamma$ and a group $G<\Aut(\Gamma)$, we write $G_x$ for the stabiliser of a vertex $x\in\Gamma$, and $G_X$ for the setwise stabiliser of a subset $X\in\Gamma$, which is to say the subgroup $G_X=\{g\in G:g(X)=X\}$.

We start with the following easy fact.
\begin{lemma}\label{lem:G/H'}
Let $\Gamma$ be a vertex-transitive graph, and let $G$ and $H$ be subgroups of $\Aut(\Gamma)$ such that $G$ normalises $H$. Then the quotient graph $\Gamma/H$ is invariant under the action of $G$ on $\Gamma$.
\end{lemma}
\begin{proof}
Given $g\in G$ and $x\in\Gamma$ the fact that $G$ normalises $H$ implies that $gH(x)=H(g(x))$.
\end{proof}
\begin{lemma}\label{lem:H.inv.for.G}
Let $\Gamma$ be a vertex-transitive graph, and let $G$ and $H$ be subgroups of $\Aut(\Gamma)$ such that $\Gamma/H$ is invariant under the action of $G$ on $\Gamma$. Define $\ph:G\to\Sym(\Gamma/H)$ via $\ph(g)(H(x))=gH(x)$. Then $\ph$ is a homomorphism $G\to\Aut(\Gamma/H)$.
\end{lemma}
\begin{proof}The map $\ph$ is trivially a homomorphism of $G$ into $\Sym(\Gamma/H)$. To see that each $\ph(g)$ is an automorphism of $\Gamma/H$, let $g\in G$, and suppose that $H(x)\sim H(y)$ in $\Gamma/H$. This means that there exists $h\in H$ such that $hx\sim y$ in $\Gamma$, and hence that
\begin{equation}\label{eq:H.inv.for.G}
ghx\sim gy
\end{equation}
in $\Gamma$. The $G$-invariance of $\Gamma/H$ means that both $\ph(g)(H(x))$ and $\ph(g)(H(y))$ are vertices of $\Gamma/H$, and since $ghx\in\ph(g)(H(x))$ and $gy\in\ph(g)(H(y))$, \eqref{eq:H.inv.for.G} means that $\ph(g)(H(x))\sim\ph(g)(H(y))$ in $\Gamma/H$, as required.
\end{proof}
If $\Gamma$ is a vertex-transitive graph, and $G$ and $H$ are subgroups of $\Aut(\Gamma)$ such that $\Gamma/H$ is invariant under the action of $G$ on $\Gamma$, we call the action of $G$ on $\Gamma/H$ defined by the map $\ph$ defined in \cref{lem:H.inv.for.G} the action \emph{induced} by the action of $G$ on $\Gamma$. We use the notation $G_{\Gamma/H}$ to mean the image group $\ph(G)\subset\Aut(\Gamma/H)$. Note that $G_{\Gamma/H}$ is isomorphic to the quotient of $G$ by the normal subgroup $\{g\in G:gH(x)=H(x)\text{ for every }x\in\Gamma\}$, which we call the \emph{kernel} of this induced action.

In light of \cref{lem:G/H',lem:H.inv.for.G}, if $\Gamma$ is a vertex-transitive graph, and $G$ and $H$ are subgroups of $\Aut(\Gamma)$ such that $G$ normalises $H$, then $G_{\Gamma/H}$ is always well defined. We assume this implicitly throughout the rest of this paper, without further mention.

\begin{lemma}\label{lem:fibres.isom}
Let $\Gamma$ be a vertex-transitive graph, and let $G$ and $H$ be subgroups of $\Aut(\Gamma)$ such that $G$ is transitive and normalises $H$. Then $G$ acts transitively on the fibres of the projection $\Gamma\to\Gamma/H$. In particular these fibres are isometric to one another (with respect to the restriction of the graph distance on $\Gamma$).
\end{lemma}
\begin{proof}
Given $x,y\in\Gamma$, the transitivity of $G$ implies that there exists $g\in G$ such that $g(y)=x$, and then the fact that $G$ normalises $H$ implies that $H(x)=H(g(y))=g(H(y))$.
\end{proof}

\begin{lemma}\label{lem:cay-ab.gen.set}
Let $\Gamma$ be a vertex-transitive graph with a distinguished vertex $e$, and suppose that $G$ is a transitive subgroup of $\Aut(\Gamma)$. Let $S=\{g\in G:d(g(e),e)\le1\}$. Then $S$ is symmetric, and an element $g\in G$ satisfies $g(e)\in B_\Gamma(e,n)$ for a given $n\ge1$ if and only if $g\in S^n$. In particular, $S$ is a symmetric generating set for $G$ containing the identity.
\end{lemma}
\begin{proof}
To see that $S$ is symmetric, note that $d(g^{-1}(e),e)=d(g^{-1}(e),g^{-1}(g(e)))=d(g(e),e)$. The fact that $g(e)\in B_\Gamma(e,n)$ for a given $n\ge1$ if and only if $g\in S^n$ is essentially \cite[Lemma 3]{woess}. We prove it by induction on $n$, the case $n=1$ being true by definition. For $n\ge2$ we have
\begin{align*}
g(e)\in B_\Gamma(e,n)&\iff g(e)\in B_\Gamma(x,1)\text{ for some }x\in B_\Gamma(e,n-1)\\
   &\iff g(e)\in B_\Gamma(h(e),1)\text{ for some }h\in S^{n-1}&\text{(by induction)}\\
   &\iff h^{-1}g\in S\text{ for some }h\in S^{n-1}&\text{(by the $n=1$ case)}\\
   &\iff g\in S^n,
\end{align*}
as required. 
\end{proof}

\begin{lemma}\label{lem:image.of.S}
Let $\Gamma$ be a vertex-transitive graph with a distinguished vertex $e$, and suppose that $H\lhd G<\Aut(\Gamma)$. Then for $g\in G$ and $n\ge0$ we have $d_{G(e)/H}(gH(e),H(e))=n$ if and only if there exists $h\in H$ such that $d_\Gamma(hg(e),e)=n$. In particular, writing $\pi:G\to G_{\Gamma/H}$ for the quotient homomorphism we have $\pi(G_e)=(G_{\Gamma/H})_{H(e)}$. If $G$ is transitive then, writing $S=\{g\in G:d_\Gamma(g(e),e)\le1\}$, for every $n\ge1$ we have
\begin{equation}\label{eq:image.of.S^n}
\pi(S^n)=\{g\in G_{\Gamma/H}:d_{\Gamma/H}(gH(e),H(e))\le n\}.
\end{equation}
\end{lemma}
\begin{proof}
We certainly have $d_{G(e)/H}(gH(e),H(e))\le d_\Gamma(gh(e),e)$ for every $h\in H$. For the converse, if $d_{G(e)/H}(gH(e),H(e))=0$ then by definition and normality of $H$ there exists $h\in H$ such that $hg(e)=e$, so we may assume that $d_{G(e)/H}(gH(e),H(e))\ge1$. Let $m$ be maximal such that there exist $x_1,\ldots,x_m,y_1,\ldots,y_m$ with $d_\Gamma(x_i,y_i)\ne0$ for every $i$ achieving the minimum in the expression \eqref{eq:quot.metric} for $d_{G(e)/H}(gH(e),H(e))$. For each $i$ pick $g_i\in G$ such that $g_i(e)=y_i$, taking $g_m=g$. For notational purposes set $g_0=1$, and note that
\begin{equation}\label{eq:sum.ind}
d_{G(e)/H}(g_\ell H(e),H(e))=\sum_{i=1}^\ell d_\Gamma(x_i,y_i)
\end{equation}
whenever $1\le\ell\le m$. By induction on $m$ we may assume that there exists $h_1\in H$ such that $d_{G(e)/H}(g_{m-1}H(e),H(e))=d_\Gamma(h_1g_{m-1}(e),e)$,
the case $m=1$ being trivial. By definition there exists $h_2\in H$ such that $x_m=h_2g_{m-1}(e)$, and so \eqref{eq:sum.ind} then gives
\begin{align*}
d_{G(e)/H}(g_mH(e),H(e))&=d_\Gamma(h_1g_{m-1}(e),e)+d_\Gamma(x_m,y_m)\\
   &=d_\Gamma(h_2g_{m-1}(e),h_2h_1^{-1}e)+d_\Gamma(h_2g_{m-1}(e),g_m(e))\\
   &\ge d_\Gamma(e,h_1h_2^{-1}g_m(e)),
\end{align*}
as required. If $G$ is transitive then \eqref{eq:image.of.S^n} follows from the first part of the lemma and \cref{lem:cay-ab.gen.set}.
\end{proof}

\begin{lemma}\label{lem:make.H.kernel}
Let $\Gamma$ be a vertex-transitive graph, and suppose that $H\lhd G<\Aut(\Gamma)$. Let $H'\lhd G$ be the kernel of the induced action of $G$ on $\Gamma/H$. Then $H\subset H'$ and $H'(x)=H(x)$ for every $x\in\Gamma$. In particular, $\Gamma/H'=\Gamma/H$ and $G_{\Gamma/H'}=G/H'$. Now suppose in addition that $G$ is transitive, write $S=\{g\in G:d(g(e),e)\le1\}$, and suppose that $H\subset S^n$ for a given $n\in\N$. Then $H'\subset S^n$.
\end{lemma}
\begin{proof}
We have $H'=\{g\in G:gH(x)=H(x)\text{ for every }x\in\Gamma\}$, which trivially satisfies $H\subset H'\lhd G$. Given $h\in H'$ and $x\in\Gamma$ we have $h(x)\in hH(x)=H(x)$ by definition of $H'$, and so $H'(x)=H(x)$. If $G$ is transitive then this combines with \cref{lem:cay-ab.gen.set} to imply that $H'\subset S^n$.
\end{proof}

\begin{lemma}\label{lem:ball.preimage}
Let $\Gamma$ be a vertex-transitive graph with a distinguished vertex $e$, let $G<\Aut(\Gamma)$ be a transitive subgroup, let $H\lhd G$, and write $\psi:\Gamma\to\Gamma/H$ for the quotient map. Write $S=\{g\in G:d(g(e),e)\le1\}$, let $k\in\N$, and suppose that $H\subset S^k$. Then
$B_\Gamma(e,m)\subset\psi^{-1}(B_{\Gamma/H}(H(e),m))\subset B_\Gamma(e,m+k)$
for every $m\in\N$.
\end{lemma}
\begin{proof}
The first inclusion is trivial. Write $\pi:G\to G_{\Gamma/H}$ for the quotient homomorphism, and note that by \cref{lem:make.H.kernel} we may assume that $H=\ker\pi$. Let $x\in\Gamma$ and $m\in\N$. The transitivity of $G$ implies that there exists $g\in G$ such that $g(e)=x$, and so
\begin{align*}
H(x)\in B_{\Gamma/H}(H(e),m)&\implies gH\in\pi(S^m)&\text{(by \cref{lem:image.of.S})}\phantom{,}\\
   &\implies g\in S^mH\\
   &\implies g\in S^{m+k}\\
   &\implies x\in B_\Gamma(e,m+k)&\text{(by \cref{lem:cay-ab.gen.set})},
\end{align*}
giving the second inclusion.
\end{proof}

\begin{lemma}\label{cor:S^n.cosets}
Let $\Gamma$ be a connected, locally finite vertex-transitive graph with a distinguished vertex $e$. Let $G<\Aut(\Gamma)$ be a transitive subgroup, and for each $x\in\Gamma$ fix some $g_x\in G$ such that $g_x(e)=x$. Let $S=\{g\in G:d(g(e),e)\le1\}$. Let $n\in\N$. Then $S^n=\bigcup_{x\in B_\Gamma(e,n)}g_xG_e$.
\end{lemma}
\begin{proof}
\cref{lem:cay-ab.gen.set} implies that $S^n=\{g\in G:g(e)\in B_\Gamma(e,n)\}=\bigcup_{x\in B_\Gamma(e,n)}g_xG_e$.
\end{proof}

\begin{lemma}\label{lem:stab.normal}
Let $\Gamma$ be a vertex-transitive graph, and let $G<\Aut(\Gamma)$. Let $x,y\in\Gamma$, and suppose that $g\in G$ satisfies $g(x)=y$. Then $G_x=(G_y)^g$. In particular, if $G$ is transitive then all of its vertex stabilisers are conjugate to one another.
\end{lemma}
\begin{proof}
Trivial.
\end{proof}

\section{The topology of pointwise convergence on $\Aut(\Gamma)$}\label{sec:top}

Woess \cite{woess} gave a beautiful and transparent proof of \cref{thm:trof} by viewing $\Aut(\Gamma)$ as a \emph{topological group}. These are introduced in detail in \cite{hew-ross}; here we simply present the properties we require.

A topological group is a group $G$ endowed with a topology with respect to which the maps
\begin{equation}\label{eq:mult}
\begin{array}{ccc}
G\times G&\to&G\\
(g,h)&\mapsto&gh
\end{array}
\end{equation}
and
\begin{equation}\label{eq:inv}
\begin{array}{ccc}
G&\to&G\\
g&\mapsto&g^{-1}
\end{array}
\end{equation}
are both continuous.

Before we describe the topology on $\Aut(\Gamma)$, we present some straightforward general properties of topological groups.
\begin{lemma}[{\cite[(4.4)]{hew-ross}}]\label{lem:top.grp.prod}
Let $G$ be a topological group, and let $A,B\subset G$. If $A$ is open then so are $AB$ and $BA$. If $A$ and $B$ are compact then so is $AB$. If $A$ is closed and $B$ is compact then $AB$ and $BA$ are closed.\qed
\end{lemma}
\begin{lemma}\label{lem:open.subgrp}
Let $G$ be a topological group, and suppose that $U<H<G$ with $U$ open. Then $H$ is open and closed.
\end{lemma}
\begin{proof}
Both the subgroup $H$ and its complement are unions of left cosets of $U$, each of which is open by the continuity of \eqref{eq:mult}.
\end{proof}

Throughout the rest of this paper, given a connected, locally finite vertex-transitive graph $\Gamma$ we endow $\Aut(\Gamma)$ with the \emph{topology of pointwise convergence}. This is defined to be the topology with respect to which a sequence $(g_n)_{n=1}^\infty$ in $\Aut(\Gamma)$ converges to $g\in\Aut(\Gamma)$ if for every $x\in\Gamma$ there exists $n_x\in\N$ such that $g_n(x)=g(x)$ for every $n\ge n_x$. Equivalently, fixing an arbitrary element $e\in\Gamma$, the topology of pointwise convergence is the topology induced by the metric $d$ on $\Aut(\Gamma)$ defined via
\begin{equation}\label{eq:metrisable}
d(g,h)=2^{-\inf\{r\ge0\,:\,\exists\,x\in B(e,r)\text{ such that }g(x)\ne h(x)\}}.
\end{equation}
Another equivalent definition is as the topology with base consisting of those subsets $U_{X,f}\subset\Aut(\Gamma)$ indexed by finite sets $X\subset\Gamma$ and isometric embeddings $f:X\to\Gamma$, and defined via
\[
U_{X,f}=\{g\in\Aut(\Gamma):g(x)=f(x)\text{ for every }x\in X\}.
\]
Note that, since the topology of pointwise convergence is metrisable via \eqref{eq:metrisable}, a subset $A$ of $\Aut(\Gamma)$ is closed if and only if it contains the limits of all convergent sequences in $A$, and compact if and only if it is sequentially compact.

We leave it to the reader to check that $\Aut(\Gamma)$ is indeed a topological group when endowed with the topology of pointwise convergence. We also leave as an exercise the following fact, which we use implicitly throughout the rest of the paper, without further mention.

\begin{lemma}\label{lem:quot/subsp}
Let $\Gamma$ be a connected, locally finite vertex-transitive graph. Let $G,H<\Aut(\Gamma)$ be subgroups such that $G$ is closed and normalises $H$. Write $H'$ for the kernel of the induced action of $G$ on $\Gamma/H$. Then the quotient topology on $G_{\Gamma/H}$ viewed as $G/H'$ agrees with the subspace topology on $G_{\Gamma/H}$ arising from $\Aut(\Gamma/H)$.
\end{lemma}

Given a subgroup $G<\Aut(\Gamma)$ we define the set $G_{x\to y}$ via $G_{x\to y}=\{g\in G:g(x)=y\}$. Note that if $g\in G_{x\to y}$ then
\begin{equation}\label{eq:Gxy=g.Gx}
G_{x\to y}=gG_x.
\end{equation}

\begin{lemma}\label{lem:stab.comp}
Let $\Gamma$ be a connected, locally finite vertex-transitive graph, let $x,y\in\Gamma$, and let $G<\Aut(\Gamma)$ be a closed subgroup. Then $G_{x\to y}$ is compact Hausdorff.
\end{lemma}
\begin{proof}
This is essentially contained in \cite[\S2]{woess}. We first show that $\Aut(\Gamma)_{x\to y}$ is compact Hausdorff. Fix an element $g\in\Aut(\Gamma)_{x\to y}$ (which is non-empty by vertex transitivity). The fact that the stabiliser $\Aut(\Gamma)_x$ is compact Hausdorff is \cite[Lemma 1]{woess}. The fact that $\Aut(\Gamma)_{x\to y}$ is compact Hausdorff then follows from \eqref{eq:Gxy=g.Gx} and the continuity of the map \eqref{eq:mult}. Since $G_{x\to y}$ is a closed subset of $\Aut(\Gamma)_{x\to y}$, it is therefore also compact Hausdorff.
\end{proof}

\begin{prop}\label{prop:tot.disc}
Let $\Gamma$ be a connected, locally finite vertex-transitive graph, and let $G<\Aut(\Gamma)$ be a closed subgroup. Then $G$ is locally compact, Hausdorff, and totally disconnected (i.e.\ the only non-empty connected subsets are the singletons).
\end{prop}
\begin{proof}
Since $G=\bigcup_{x,y\in\Gamma}G_{x\to y}$, the local compactness of $G$ is immediate from \cref{lem:stab.comp}. To see that $G$ is totally disconnected, suppose that $A\subset G$ contains two distinct elements $g,h$. This means that there exist $x,y\in\Gamma$ such that $g(x)=y\ne h(x)$. The set $G_{x\to y}$ is open in $G$ by definition, whilst $G\setminus G_{x\to y}$ is open in $G$ by \cref{lem:stab.comp}. Since $g\in G_{x\to y}$ and $h\in G\setminus G_{x\to y}$, this implies that $A$ is not connected.
\end{proof}

\begin{lemma}\label{lem:discrete}
Let $\Gamma$ be a non-empty, connected, locally finite vertex-transitive graph, and let $G<\Aut(\Gamma)$ be a closed subgroup. Then the following are equivalent.
\begin{enumerate}[label=(\roman*)]
\item\label{item:discrete}The group $G$ is discrete.
\item\label{item:all.stabs.finite}For every $x\in\Gamma$ the stabiliser $G_x$ is finite.
\item\label{item:stab.e.finite}There exists $x\in\Gamma$ such that $G_x$ is finite.
\end{enumerate}
\end{lemma}
\begin{proof}
It is trivial that \ref{item:all.stabs.finite} implies \ref{item:stab.e.finite}. Lemma \ref{lem:stab.comp} implies that every stabiliser $G_x$ is compact, which means in particular that if $G$ is discrete then $G_x$ must be finite, and hence that \ref{item:discrete} implies \ref{item:all.stabs.finite}. Finally, suppose that \ref{item:stab.e.finite} holds, which is to say that there exists $x\in\Gamma$ such that $G_x$ is finite. We first claim that all convergent sequences in $G$ are eventually constant. Indeed, if the sequence $(g_n)_{n=1}^\infty\subset G$ satisfies $g_n\to g$ then by definition there exists $n_0\in\N$ such that $g_n\in G_{x\to g(x)}$ for every $n\ge n_0$. On the other hand, \eqref{eq:Gxy=g.Gx} and the finiteness of $G_x$ imply that $G_{x\to g(x)}$ is finite, and so the sequence $(g_n)_{n=1}^\infty$ takes only finitely many values, and so it must eventually be constant, as claimed. This implies that all subsets of $G$ are closed, and so $G$ is discrete and \ref{item:discrete} holds.
\end{proof}

\begin{lemma}\label{lem:comp.closure}
Let $\Gamma$ be a non-empty connected, locally finite vertex-transitive graph, and let $U\subset\Aut(\Gamma)$. Then the following are equivalent.
\begin{enumerate}[label=(\roman*)]
\item\label{item:comp.cl}The set $U$ has compact closure.
\item\label{item:forall}The orbit $U(x)$ is finite for every $x\in\Gamma$.
\item\label{item:exists}There exists some $x_0\in\Gamma$ such that the orbit $U(x_0)$ is finite.
\end{enumerate}
\end{lemma}
\begin{proof}
The equivalence of \ref{item:comp.cl} and \ref{item:forall} is \cite[Lemma 2]{woess}, and \ref{item:forall} trivially implies \ref{item:exists}. As remarked in \cite[\S2]{woess}, the fact that \ref{item:exists} implies \ref{item:forall} is immediate from the local finiteness of $\Gamma$.
\end{proof}

A \emph{Haar measure} $\mu$ on a locally compact topological group $G$ is a certain measure on the Borel sigma algebra of $G$, the properties of which include that
\begin{enumerate}[label=(\arabic*)]
\item$\mu(V)<\infty$ if $V$ is compact,
\item$\mu(U)>0$ if $U$ is open and nonempty, and
\item$\mu(gA)=\mu(A)$ for every Borel set $A\subset G$ and every $g\in G$.
\end{enumerate}
See \cite[\S15]{hew-ross} for a detailed introduction to Haar measures. Every locally compact $T_0$ topological group admits a Haar measure \cite[(15.5) \& (15.8)]{hew-ross}. In particular, \cref{prop:tot.disc} implies that if $\Gamma$ is a connected, locally finite vertex-transitive graph and $G<\Aut(\Gamma)$ is a closed subgroup then $G$ admits a Haar measure.

\begin{lemma}\label{cor:growth.inherit}
Let $\Gamma$ be a connected, locally finite vertex-transitive graph with distinguished vertex $e$, let $G<\Aut(\Gamma)$ be a closed subgroup, and set $S=\{g\in G:d(g(e),e)\le1\}$. Then for each $n\in\N$ the set $S^n$ is an open compact generating set for $G$ containing the identity. Moreover, if $\mu$ is a Haar measure on $G$ then $\mu(S^n)=\mu(G_e)\beta_\Gamma(n)$ for every $n\in\N$.
\end{lemma}
\begin{proof}
The set $S$ is trivially open and closed, and hence compact by \cref{lem:comp.closure}. \cref{lem:top.grp.prod} therefore implies that $S^n$ is open and compact for every $n\in\N$, and \cref{cor:S^n.cosets} then implies that $\mu(S^n)=\mu(G_e)\beta_\Gamma(n)$.
\end{proof}

\section{Quasi-isometries}

In this section we briefly record the basic properties of quasi-isometries that we will need. We start by recording the following standard general lemma.
\begin{lemma}\label{lem:QI.elem}
Let $C,D\ge1$ and $K,L\ge0$. Let $X,Y,Z$ be metric spaces, and suppose that $f:X\to Y$ is a $(C,K)$-quasi-isometry and $g:Y\to Z$ is a $(D,L)$-quasi-isometry. Then the following hold.
\begin{enumerate}[label=(\roman*)]
\item The composition $g\circ f:X\to Z$ is a $(CD,DK+2L)$-quasi isometry.
\item Define a map $\hat{f}:Y\to X$ by picking, for each $y\in Y$, an arbitrary element $\nu(y)\in f(X)$ of distance at most $K$ from $y$, and then setting $\hat{f}(y)$ to be an arbitrary element of $f^{-1}(\nu(y))$. Then $\hat{f}$ is a $(C,3CK)$-quasi-isometry.
\end{enumerate} 
\end{lemma}
\begin{proof}
These statements are easily verified by direct calculation.
\end{proof}

We now come onto two facts about quasi-isometries that are specific to the setting of this paper.
\begin{lemma}\label{lem:quotient.QI}
Let $\Gamma$ be a connected locally finite vertex-transitive graph with a distinguished vertex $e$, let $G<\Aut(\Gamma)$ be a transitive subgroup, and let $H\lhd G$. Write $S=\{g\in G:d(g(e),e)\le1\}$, let $k\in\N$, and suppose that $H\subset S^k$. Then the quotient map $\psi:\Gamma\to\Gamma/H$ is a $(1,k)$-quasi-isometry.
\end{lemma}
\begin{proof}
The map $\psi$ is surjective and satisfies $d(\psi(x),\psi(y))\le d(x,y)$ for every $x,y\in\Gamma$, so it remains to show that $d(\psi(x),\psi(y))\ge d(x,y)-k$ for every $x,y\in\Gamma$. However, this follows from the fact that $\psi^{-1}(B_{\Gamma/H}(\psi(x),d(x,y)-k-1)\subset B_\Gamma(x,d(x,y)-1)$, which is a consequence of \cref{lem:ball.preimage}.
\end{proof}

\begin{lemma}\label{lem:aut.QI}
Let $\Gamma$ be a connected locally finite vertex-transitive graph with a distinguished vertex $e$, and let $G<\Aut(\Gamma)$ be a transitive subgroup. Write $S=\{g\in G:d(g(e),e)\le1\}$. Then the map
\[
\begin{array}{ccccc}
f&:&\Cay(G,S)&\to&\Gamma\\
 &&g&\mapsto&g(e)
\end{array}
\]
is a $(1,1)$-quasi-isometry.
\end{lemma}
\begin{proof}
Let $g,h\in G$ with $g\ne h$. If $g(e)\ne h(e)$ then $d(f(g),f(h))=d(g^{-1}h(e),e)=d(g,h)$ by \cref{lem:cay-ab.gen.set}, whilst if $g(e)=h(e)$ then $d(f(g),f(h))=0$ and $d(g,h)=1$ by definition. The map $f$ is surjective by transitivity, and so the lemma is proved.
\end{proof}
\begin{remarks*}\mbox{}
\begin{enumerate}[label=(\arabic*)]
\item \cref{lem:aut.QI} shows that a connected, locally finite vertex-transitive graph is always quasi-isometric to some Cayley graph, just not necessarily a locally finite one. The point about \cref{thm:trof.geom,thm:main.geom}, of course, is that in each case $\Gamma$ is quasi-isometric to a locally finite Cayley graph.
\item The graph $\Gamma$ appearing in \cref{lem:aut.QI} can be viewed as a certain quotient of $\Cay(G,S)$ called a \emph{Cayley--Abels graph} (see \cite[\S2.3]{bft} for details), and \cref{lem:aut.QI} is then actually a special case of \cref{lem:quotient.QI}. 
\end{enumerate}
\end{remarks*}

\section{Approximate groups}
Approximate groups are, roughly speaking, subsets of groups that are `approximately closed' under the group operation. Precisely, given $K\ge1$, a subset $A$ of a group $G$ is said to be a \emph{$K$-approximate subgroup of $G$}, or simply a \emph{$K$-approximate group}, if it is symmetric and contains the identity and there exists a set $X\subset G$ of cardinality at most $K$ such that $A^2\subset XA$. The reader will find a detailed introduction to approximate groups, including detailed motivation for this definition, in \cite{book}, but the immediate relevance to \cref{thm:main} comes from the following result.
\begin{prop}\label{prop:tripling}
Let $A$ be a symmetric open precompact set in a locally compact group $G$ with a Haar measure $\mu$. Suppose that $\mu(A^3)\le K\mu(A)$. Then $A^2$ is an open precompact $K^3$-approximate group.
\end{prop}
\cref{prop:tripling} is essentially due to Ruzsa in the discrete setting. Tao \cite{tao.product.set} showed that the argument generalises to the setting of a locally compact group with a Haar measure. There is a global assumption in Tao's paper that the Haar measure is bi-invariant, which is not necessarily the case in the present paper; to reassure the reader that the relevant arguments do not require that assumption we reproduce them here.
\begin{lemma}[Ruzsa's triangle inequality {\cite[Lemma 3.2]{tao.product.set}}]\label{lem:triangle}
Let $U,V,W$ be open precompact sets in a locally compact group $G$ with a Haar measure $\mu$. Then $\mu(UW^{-1})\mu(V^{-1})\le\mu(UV^{-1})\mu(VW^{-1})$.
\end{lemma}
\begin{proof}
We follow Tao in noting that
\begin{align*}
\mu(UV^{-1})\mu(VW^{-1})&=\int_G1_{UV^{-1}}(x)\mu(VW^{-1})\diff\mu(x)\\
   &=\int_G1_{UV^{-1}}(x)\mu(xVW^{-1})\diff\mu(x)\\
   &=\int_G\int_G1_{UV^{-1}}(x)1_{VW^{-1}}(x^{-1}y)\diff\mu(x)\diff\mu(y)&\text{(by Fubini's theorem)}\\
   &\ge\int_{UW^{-1}}\int_G1_{UV^{-1}}(x)1_{VW^{-1}}(x^{-1}y)\diff\mu(x)\diff\mu(y).
\end{align*}
Picking, for each $y\in UW^{-1}$, elements $u_y\in U$ and $w_y\in W$ such that $y=u_yw_y^{-1}$, we therefore have
\begin{align*}
\mu(UV^{-1})\mu(VW^{-1})&\ge\int_{UW^{-1}}\int_G1_{UV^{-1}}(x)1_{VW^{-1}}(x^{-1}u_yw_y^{-1})\diff\mu(x)\diff\mu(y)\\
   &\ge\int_{UW^{-1}}\int_G1_{u_yV^{-1}}(x)1_{VW^{-1}}(x^{-1}u_yw_y^{-1})\diff\mu(x)\diff\mu(y)\\
   &=\int_{UW^{-1}}\int_G1_{u_yV^{-1}}(x)\diff\mu(x)\diff\mu(y)\\
   &=\mu(UW^{-1})\mu(V^{-1}),
\end{align*}
as required.
\end{proof}
\begin{lemma}\label{lem:higher.prods}
Let $A$ be a symmetric open precompact set in a locally compact group $G$ with a Haar measure $\mu$. Suppose that $\mu(A^3)\le K\mu(A)$. Then $\mu(A^m)\le K^{m-2}\mu(A)$ for every $m\ge3$.
\end{lemma}
\begin{proof}
Given $m\ge4$, applying \cref{lem:triangle} with $U=A^2$, $V=A$ and $W=A^{m-2}$ implies that $\mu(A^m)\mu(A)\le\mu(A^3)\mu(A^{m-1})$, from which the desired result follows by induction and the $m=3$ case.
\end{proof}

\begin{proof}[Proof of \cref{prop:tripling}]
\cref{lem:top.grp.prod} implies that $A^2$ is open and precompact. \cref{lem:higher.prods} implies that $\mu(A^5)\le K^3\mu(A)$, and then \cite[Lemma 5.4]{Ca} implies that there exists a set $X\subset A^4$ of size at most $K^3$ such that $A^4\subset XA^2$ (see also the proof of \cite[Lemma 5.5]{Ca}).
\end{proof}

Most of the early work on approximate groups treated the special case of finite approximate subgroups of discrete groups. Breuillard, Green and Tao \cite{bgt} have given a rough classification of such approximate groups; the main result we need from their work is the following variant of Theorem \ref{thm:bgt.gromov}.
\begin{theorem}[Breuillard--Green--Tao {\cite[Corollary 11.2\footnote{Although the statement of \cite[Corollary 11.2]{bgt} does not mention that $G_1$ is normal in $G$, this property is stated explicitly in the proof} \& Remark 11.4]{bgt}}]\label{thm:bgt.grom.ag}
Let $K\ge1$. Then there exists $n_0=n_0(K)\in\N$ such that the following holds. Let $G$ be a group generated by a finite symmetric set $S$ containing the identity, and suppose that $n\ge n_0$ and $S^n$ is a $K$-approximate group. Then there are subgroups $H,G_1\lhd G$ with $H\subset S^{4n}\cap G_1$ such that $G_1$ has index $O_K(1)$ in $G$, and such that $G_1/H$ is nilpotent of rank and step $O_K(1)$.
\end{theorem}

We also need the following result, which is based on work of the second author on finite nilpotent approximate groups. In it, and throughout the paper, given two group elements $g,h$ write $h^g=g^{-1}hg$. Given two groups $H,G$ with $H<G$ we write $H^G$ for the normal closure of $H$ in $G$, which is to say the subgroup $H^G=\langle h^g:h\in H,g\in G\rangle$.
\begin{prop}\label{prop:normal.closure}
Let $G$ be a group generated by a finite symmetric set $S$ containing the identity, and suppose that $N\lhd G$ is a nilpotent normal subgroup of step $s$ and index $k$. Suppose further that $S^n$ is a $K$-approximate group, and that $H$ is a subgroup of $G$ contained in $S^n\cap N$. Then $H^G\subset S^{{nK^{O_s(k)}}}$.
\end{prop}
The main ingredient in the proof of \cref{prop:normal.closure} is the following result.
\begin{lemma}[{\cite[Proposition 7.3]{nilp.frei}}]\label{lem:norm.closure}
Let $N$ be an $s$-step nilpotent group generated by a $K$-approximate group $A$. Let $H\subset A$ be a subgroup of $N$. Then $H^N\subset A^{K^{O_s(1)}}$.\qed
\end{lemma}

We also use the following standard lemmas.

\begin{lemma}\label{lem:cosets.increasing}\label{lem:coset.reps.ball}
Let $G$ be a group with symmetric generating set $S$ containing the identity, and let $H$ be a subgroup of index at least $m\in\N$ in $G$. Then $S^{m-1}$ has non-empty intersection with at least $m$ distinct left-cosets of $H$.
\end{lemma}
\begin{proof}
If $S^{n+1}H=S^nH$ for a given $n\ge0$ then it follows by induction that $S^rH=S^nH$ for every $r\ge n$, and hence that $G=S^nH$. We may therefore assume that $H\subsetneqq SH\subsetneqq S^2H\subsetneqq\ldots\subsetneqq S^{m-1}H$. This implies that the number of left-cosets of $H$ having non-empty intersection with $S^n$ is strictly increasing for $n=0,1,\ldots,m-1$, and so the lemma is proved.
\end{proof}

\begin{lemma}\label{lem:fi.normal}Let $G$ be a group and let $H<G$ have index $k\in\N$ in $G$. Then there exists a subgroup $H'<H$ with $H'\lhd G$ such that $[G:H']\le k!$.
\end{lemma}
\begin{proof}
The action of $G$ on $G/H$ defined by $g(xH)=(gx)H$ induces a homomorphism of $G$ into $\Sym(G/H)$. Take $H'$ to be the kernel of this homomorphism, noting that $H'$ is normal, contained in $H$, and of index at most $|\Sym(G/H)|=k!$ in $G$.
\end{proof}

\begin{lemma}\label{lem:H^G.in.S^C}
Let $G$ be a group with symmetric generating set $S$ containing the identity, and suppose that $H_0\lhd N\lhd G$, that $N$ has index at most $k$ in $G$, and that $H_0\subset S^r$. Then $H_0^G\subset S^{kr+2k^2}$. Moreover, if $G$ is a topological group and $H_0$ is compact then so is $H_0^G$.
\end{lemma}
\begin{proof}
Applying \cref{lem:coset.reps.ball}, let $g_1,\ldots,g_k$ be a complete set of coset representatives for $N$ in $G$ such that $g_i\in S^k$ for each $i$. Since each group $H_0^{g_i}$ is a normal subgroup of $N$ we have $H_0^G=\prod_{i=1}^kH_0^{g_i}$. This implies that $H_0^G\subset S^{kr+2k^2}$, as required, and combined with \cref{lem:top.grp.prod} it implies that if $H_0$ is compact then so is $H_0^G$.
\end{proof}

\begin{proof}[Proof of \cref{prop:normal.closure}]
If $K<2$ then $S^n$ is a subgroup, and hence equal to $G$, and the proposition is trivial. We may therefore assume that $K\ge2$, which confers the minor notational convenience that any quantity bounded by $O(K)$ is also bounded by $K^{O(1)}.$ It follows from \cite[Lemma 4.2]{bt} that $N=\left\langle S^{2kn-1}\cap N\right\rangle$, and from \cite[Lemma 2.10 (ii)]{nilp.frei} that $S^{2kn-1}\cap N$ is a $K^{4k}$-approximate group. \cref{lem:norm.closure} therefore implies that
$H^N\subset S^{knK^{O_s(k)}}$, which simplifies to $H^N\subset S^{nK^{O_s(k)}}$. Applying \cref{lem:H^G.in.S^C} with $H_0=H^N$ and $r=nK^{O_s(k)}$ gives $H^G\subset S^{{k^2+nK^{O_s(k)}}}\subset S^{{nK^{O_s(k)}}}$.
\end{proof}

A compactly generated group $G$ with Haar measure $\mu$ is said to have \emph{polynomial growth} if there exist a compact symmetric neighbourhood $S$ of the identity and constants $C,d\ge0$ such that $\mu(S^n)\le Cn^d$ for every $n\in\N$. Losert \cite[Theorem 2]{losert} showed that if $G$ is a locally compact Hausdorff group with polynomial growth then there is a compact normal subgroup $H\lhd G$ such that $G/K$ is a Lie group. Carolino \cite{Ca} obtained the following refinement of Losert's theorem, in a similar spirit to the Breuillard--Green--Tao refinement of Gromov's theorem.
\begin{theorem}[Carolino {\cite[Theorem 1.9]{Ca}}]\label{thm:carolino}
Let $A$ be an open precompact $K$-approximate subgroup of a locally compact Hausdorff group $G$. Then there exist a subgroup $L<G$ and a compact normal subgroup $H\lhd L$ such that $H\subset A^4$, such that $A$ is contained in the union of at most $O_K(1)$ left-cosets of $L$, and such that $L/H$ is a Lie group of dimension at most $O_K(1)$.\qed
\end{theorem}

Using a similar argument to that used by Breuillard, Green and Tao to deduce \cref{thm:bgt.grom.ag} from their work on finite approximate groups, one can deduce the following result from \cref{thm:carolino}.
\begin{corollary}\label{cor:caro.grom}
Let $K\ge1$. Then there exists $n_0=n_0(K)\in\N$ such that the following holds. Let $G$ be a locally compact Hausdorff group, let $S$ be an open compact symmetric generating set containing the identity, and suppose that $n\ge n_0$ and $S^n$ is a $K$-approximate group. Then there are subgroups $H,L\lhd G$ with $H\subset S^{O_K(n)}\cap L$ such that $H$ is compact and $L$ is open and has index $O_K(1)$ in $G$, and such that $L/H$ is a Lie group of dimension at most $O_K(1)$.
\end{corollary}

\begin{proof}
Let $n_0$ be the bound on the number of left-cosets appearing in the conclusion of \cref{thm:carolino}, and suppose that $n\ge n_0$ and $S^n$ is a $K$-approximate group. \cref{thm:carolino} implies that there exist a subgroup $L_0<G$ and a compact normal subgroup $H_0\lhd L_0$ such that $H_0\subset S^{4n}$, such that $S^n$ is contained in the union of at most $n_0$ left-cosets of $L_0$, and such that $L_0/H_0$ is a Lie group of dimension at most $O_K(1)$. \cref{lem:cosets.increasing} implies that $L_0$ has index at most $n_0$ in $G$. \cref{lem:fi.normal} then implies that there exists a subgroup $L<L_0$ of index $O_K(1)$ that is normal in $G$. Since the quotient $G/L$ is finite, and hence discrete, $L$ is open. Taking $H=(L\cap H_0)^G$ then satisfies the corollary by \cref{lem:H^G.in.S^C}.
\end{proof}

\section{Proof of the main result}\label{sec:proof}

In this section we prove \cref{thm:main}, starting with the following preliminary generalisation of Theorem \ref{thm:trof}.

\begin{theorem}\label{thm:prelim}
Given $K\ge1$ there exists $n_0=n_0(K)$ such that the following holds. Suppose $\Gamma$ is a connected, locally finite vertex-transitive graph satisfying $\beta_\Gamma(3n)\le K\beta_\Gamma(n)$ for some $n\ge n_0$, and let $G\le\Aut(\Gamma)$ be a closed transitive subgroup. Then the set $S=\{g\in G:d(g(e),e)\le1\}$ is a compact open generating set for $G$ containing the identity, and $S^{2n}$ is a compact open $K^3$-approximate group. Furthermore, there is a compact open normal subgroup $H\lhd G$ such that
\begin{enumerate}[label=(\roman*)]
\item\label{item:proj.diam.prelim} every fibre of the projection $\Gamma\to\Gamma/H$ has diameter at most $O_K(n)$;
\item\label{item:H.is.ker} $G_{\Gamma/H}=G/H$;
\item\label{item:v.nilp.prelim} $G_{\Gamma/H}$ is a finitely generated group with a nilpotent normal subgroup of rank, step and index at most $O_K(1)$; and
\item\label{item:fin.stabs}every vertex stabiliser of the action of $G_{\Gamma/H}$ on $\Gamma/H$ is finite.
\end{enumerate}
\end{theorem}
\begin{remark*}
By \cref{lem:discrete}, conclusion \ref{item:fin.stabs} of \cref{thm:prelim} is equivalent to saying that $G_{\Gamma/H}$ is discrete, which by conclusion \ref{item:H.is.ker} is in turn equivalent to $H$ being open.
\end{remark*}

\begin{proof}
Woess \cite[Theorem 1]{woess} gave a short proof of Theorem \ref{thm:trof} using Gromov's and Losert's theorems; much of our proof of \cref{thm:prelim} is essentially Woess's argument with Gromov's theorem replaced by Theorem \ref{thm:bgt.gromov} and Losert's theorem replaced by Theorem \ref{thm:carolino}.

Fixing some Haar measure $\mu$ on $G$, \cref{cor:growth.inherit} implies that $S^m$ is a compact open generating set for $G$ containing the identity for every $m\in\N$, and that $\mu(S^{3n})\le K\mu(S^n)$. \cref{prop:tripling} then implies that $S^{2n}$ is a compact open $K^3$-approximate group, as required. \cref{cor:caro.grom} then implies that there exist subgroups $H_0,L\lhd G$, with $H_0\subset S^{O_K(n)}\cap L$ compact, such that $L$ has finite index in $G$ and $L/H_0$ is a Lie group.

The kernel $\{g\in G:gH_0(x)=H_0(x)\text{ for every }x\in\Gamma\}$ of the induced action of $G$ on $\Gamma/H_0$ is clearly closed, so $G_{\Gamma/H_0}$ is a quotient of $G/H_0$ by a closed subgroup. In particular, this means that $L_{\Gamma/H_0}$ is the quotient of $L/H_0$ by a closed subgroup, and hence that $L_{\Gamma/H_0}$ is a Lie group. \cref{lem:comp.closure} implies that the orbits of $H_0$ on $\Gamma$ are finite, and hence that $\Gamma/H_0$ is locally finite. \cref{prop:tot.disc} therefore implies that $\Aut(\Gamma/H_0)$ is totally disconnected, and hence in particular that $L_{\Gamma/H_0}$ is. Thus $L_{\Gamma/H_0}$ is a totally disconnected Lie group, and hence discrete. Since $L_{\Gamma/H_0}$ has finite index in $G_{\Gamma/H_0}$, it follows that $G_{\Gamma/H_0}$ is also discrete.

Write $\pi:G\to G_{\Gamma/H_0}$ for the quotient homomorphism. Since $G_{\Gamma/H_0}$ is discrete, $\ker\pi$ is open and \cref{lem:discrete} implies that $G_{\Gamma/H_0}$ acts on $\Gamma/H_0$ with finite vertex stabilisers. Lemmas \ref{lem:image.of.S} and \ref{cor:S^n.cosets} therefore imply that $\pi(S)$ is finite. In particular, $\pi(S)$ is a finite symmetric generating set for $G_{\Gamma/H_0}$ containing the identity such that $\pi(S^{2n})$ is a $K^3$-approximate group. \cref{thm:bgt.grom.ag} therefore implies that there are subgroups $H_1,N_1\lhd G_{\Gamma/H_0}$ with $H_1\subset\pi(S^{8n})\cap N_1$ such that $N_1/H_1$ is nilpotent of rank and step $O_K(1)$, and such that $N_1$ has index $O_K(1)$ in $G_{\Gamma/H_0}$.

Set $H_2=\pi^{-1}(H_1)$ and $N=\pi^{-1}(N_1)$, and then let $H$ be the kernel of the induced action of $G$ on $\Gamma/H_2$. Note that \cref{lem:make.H.kernel} then implies that $H_2<H\lhd G$ and $H\subset S^{O_K(n)}$, and that \ref{item:H.is.ker} holds. Moreover, since $\ker\pi$ is an open subgroup of $H$, \cref{lem:open.subgrp} implies that $H$ is open and closed, and then \cref{lem:cay-ab.gen.set,lem:comp.closure} imply that $H$ is compact.

\cref{lem:cay-ab.gen.set} implies that $H(e)$ has diameter at most $O_K(n)$, and then \cref{lem:fibres.isom} implies that every other fibre $H(x)$ with $x\in\Gamma$ also has diameter at most $O_K(n)$, so \ref{item:proj.diam.prelim} holds. Note that $N_{\Gamma/H}$ is a quotient of $N_1/H_1$, and hence nilpotent of rank and step at most $O_K(1)$, and also that $N_{\Gamma/H}$ is normal and has index at most $O_K(1)$ in $G_{\Gamma/H}$. Moreover, $G_{\Gamma/H}$ is a quotient of $G_{\Gamma/H_0}$, and hence finitely generated and discrete. Thus \ref{item:v.nilp.prelim} holds, as does \ref{item:fin.stabs} by \cref{lem:discrete}.
\end{proof}

\begin{remark*}
Applying \cref{thm:prelim} to the closure of $G$ allows one to deduce a similar result without the assumption that $G$ is closed. This is also noted by Woess in his proof of \cref{thm:trof}. We perform such a reduction in detail in the proof of the stronger \cref{thm:main}, below.
\end{remark*}

\begin{proof}[Proof of \cref{thm:main}]
Since $G$ is transitive its closure $\overline G$ in $\Aut(\Gamma)$ certainly is. \cref{thm:prelim} therefore implies that $S_0=\{g\in\overline G:d(g(e),e)\le1\}$ is a symmetric generating set for $G$ containing the identity, that $S_0^{2n}$ is a $K^3$-approximate group, and that there exists a compact open subgroup $H_1\lhd\overline G$ such that
\begin{enumerate}[label=(\alph*)]
\item\label{item:a} every fibre of the projection $\Gamma\to\Gamma/H_1$ has diameter at most $O_K(n)$;
\item\label{item:b} $\overline G_{\Gamma/H_1}=\overline G/H_1$;
\item\label{item:c} $\overline G_{\Gamma/H_1}$ is a finitely generated group with a nilpotent normal subgroup $N_1$ of rank, step and index at most $O_K(1)$
; and
\item\label{item:d} every vertex stabiliser of the action of $\overline G_{\Gamma/H_1}$ on $\Gamma/H_1$ is finite.
\end{enumerate}
Write $\pi:\overline G\to\overline G/H_1$ for the quotient homomorphism, and set $N=\pi^{-1}(N_1)$.

First we claim that $\pi(S_0)$ is finite. Indeed, setting $S_1=\{g\in\overline G/H_1:d_{\Gamma/H_1}(gH_1(e),H_1(e))\le1\}$, \ref{item:b} and \cref{lem:image.of.S} imply that $\pi(S_0)=S_1$, and \ref{item:b}, \ref{item:d} and \cref{cor:S^n.cosets} imply that
\begin{equation}\label{eq:S_1.finite}
|S_1|<\infty.
\end{equation}
Thus $\pi(S_0)$ is a finite symmetric generating set for $\overline G/H_1$ containing the identity.

Since $S_0^{2n}$ is a $K^3$-approximate group, so is $\pi(S_0)^{2n}$. Since $\overline G_e\subset S_0$ by definition, applying \cref{prop:normal.closure} modulo $H_1$ therefore implies that
\begin{equation}\label{eq:G_e^G.in.S.H_1}
(\overline G_e\cap N)^{\overline G}\subset S_0^{O_K(n)}H_1.
\end{equation}
Set $H_2=(\overline G_e\cap N)^{\overline G}H_1$. Property \ref{item:a} implies in particular that $H_1(e)\subset B_\Gamma(e,O_K(n))$, and so \cref{lem:cay-ab.gen.set} implies that $H_1\subset S_0^{O_K(n)}$. This combines with  \eqref{eq:G_e^G.in.S.H_1} to imply that $H_2\subset S_0^{O_K(n)}$.
Next let $H_3$ be the kernel of the induced action of $\overline G$ on $\overline G_{\Gamma/H_2}$, noting that
\begin{equation}\label{eq:H_2.in.S^n}
H_3\subset S_0^{O_K(n)}
\end{equation}
and
\begin{equation}\label{eq:G/H_2}
\overline G_{\Gamma/H_3}=\overline G/H_3
\end{equation}
by \cref{lem:make.H.kernel}. Now set $H=H_3\cap G$, noting that since $H_3\lhd\overline G$ we also have $H\lhd G$, as required. \cref{lem:cay-ab.gen.set} and \eqref{eq:H_2.in.S^n} imply that $H(e)$ has diameter at most $O_K(n)$, and then \cref{lem:fibres.isom} implies that all fibres of the projection $\Gamma\to\Gamma/H$ have diameter equal to that of the fibre $H(e)$, and hence at most $O_K(n)$. This satisfies \ref{item:i}.

Since $H_3$ contains the subgroup $H_1$, which is open in $\overline G$, \cref{lem:open.subgrp} implies that $H_3$ is open in $\overline G$. It follows that $H$ is dense in $H_3$, and hence that $H(x)=H_3(x)$ for every $x\in\Gamma$. In particular, this implies that
\begin{equation}\label{eq:Gam/H=Gam/H_3}
\Gamma/H=\Gamma/H_3,
\end{equation}
and hence that
\begin{equation}\label{eq:G/H<G/H_2}
G_{\Gamma/H}=G_{\Gamma/H_3}.
\end{equation} 
It follows from \eqref{eq:G/H_2} that $G_{\Gamma/H_3}=G/H$, and so \ref{item:ii} is satisfied.

Since $H_1\subset H_3$, the group $\overline G_{\Gamma/H_3}$ is a quotient of $\overline G_{\Gamma/H_1}$. It therefore follows from \ref{item:c} that $\overline G_{\Gamma/H_3}$ has a nilpotent subgroup of rank, step and index at most $O_K(1)$. Since $G_{\Gamma/H_3}\subset\overline G_{\Gamma/H_3}$, \eqref{eq:G/H<G/H_2} therefore implies that $G_{\Gamma/H}$ also has a nilpotent subgroup of rank, step and index at most $O_K(1)$ (this uses the well-known fact that if $N$ is a nilpotent group of rank $r$ and step $s$ then every subgroup of $N$ has rank $O_{r,s}(1)$). This satisfies \ref{item:iii}.

Set $S_2=\{g\in\overline G_{\Gamma/H_3}:d_{\Gamma/H_3}(g(H_3(e)),H_3(e))\le1\}$, and write $\tau:\overline G/H_1\to\overline G/H_3$ for the quotient homomorphism. It follows from \eqref{eq:G/H_2} and \cref{lem:image.of.S} that $S_2=\tau(S_1)$, and hence from \eqref{eq:S_1.finite} that $S_2$ is finite. It follows from \eqref{eq:G/H<G/H_2} that $S\subset S_2$, and so $S$ is also finite. \cref{lem:cay-ab.gen.set} implies that $S$ is a symmetric generating set for $G_{\Gamma/H}$ containing the identity, and so \ref{item:iv} is satisfied.

\cref{lem:image.of.S} combines with \eqref{eq:G/H_2} to imply that the stabiliser $(\overline G_{\Gamma/H_3})_{H_3(e)}=\tau\circ\pi(\overline G_e)$. Since $\ker(\tau\circ\pi)$ contains $\overline G_e\cap N$, this implies that $(\overline G_{\Gamma/H_3})_{H_3(e)}$ is a homomorphic image of $\overline G_e/(\overline G_e\cap N)$. This in turn is isomorphic to a subgroup of $\overline G/N$, which has size $O_K(1)$ by \ref{item:c}, and so $|(\overline G_{\Gamma/H_3})_{H_3(e)}|\ll_K1$. It follows from \eqref{eq:Gam/H=Gam/H_3} that $(G_{\Gamma/H})_{H(e)}\subset (\overline G_{\Gamma/H_3})_{H_3(e)}$, and so \cref{lem:stab.normal} implies that \ref{item:v} is satisfied.

Finally, \ref{item:i} and \cref{lem:cay-ab.gen.set} combine with \cref{lem:quotient.QI} to show that the quotient map $\Gamma\to\Gamma/H$ is a $(1,O_K(n))$-quasi-isometry, whilst \cref{lem:aut.QI} implies that there is a $(1,1)$-quasi-isometry $\Cay(G_{\Gamma/H},S)\to\Gamma/H$, so \ref{item:vi} follows from \cref{lem:QI.elem}.
\end{proof}

\section{The structure of vertex-transitive graphs of polynomial growth}
In this section we prove \cref{cor:trof,cor:trof.finite}. We start by recording the following easy and well-known lemma.
\begin{lemma}\label{lem:poly.pigeon}
Let $d>0$, let $q\in\N$, and let $\alpha,\beta\in(0,1)$ be such that $\alpha<\beta$. Then there exists $K=K_{d,q,\alpha,\beta}$ such that if $\Gamma$ is a graph with a distinguished vertex $e$, and
\begin{equation}\label{eq:graph.poly.growth}
|B_\Gamma(e,n)|\le n^d|B_\Gamma(e,1)|
\end{equation}
for some $n\ge q^{2/(\beta-\alpha)}$, then there exists $m\in\N$ with $\lfloor n^\alpha\rfloor\le m\le n^\beta$ such that $|B_\Gamma(e,qm)|\le K|B_\Gamma(e,m)|$. Indeed, we may take $K=q^{2d/(\beta-\alpha)}$.
\end{lemma}
\begin{proof}
The assumption that $n\ge q^{2/(\beta-\alpha)}$ implies that $n^{(\beta-\alpha)/2}\ge q$, and in particular that there exists some $r\in\N$
such that
\begin{equation}\label{eq:r>log.q}
n^{(\alpha+\beta)/2}\le q^r\lfloor n^\alpha\rfloor\le n^\beta.
\end{equation}
Suppose the lemma does not hold for a given value of $K$, and in particular that
$|B_\Gamma(e,q^k\lfloor n^\alpha\rfloor)|>K|B_\Gamma(e,q^{k-1}\lfloor n^\alpha\rfloor)|$
for every $k=1,\ldots,r$, and hence that $|B_\Gamma(e,\lfloor n^\beta\rfloor)|>K^r|B_\Gamma(e,\lfloor n^\alpha\rfloor)|$. It follows from \eqref{eq:r>log.q} that $r\ge\frac{1}{2}(\beta-\alpha)\log_qn$, and hence that
\begin{align*}
|B_\Gamma(e,n)|&\ge|B_\Gamma(e,\lfloor n^\beta\rfloor)|\\
    &>K^{\frac{1}{2}(\beta-\alpha)\log_qn}|B_\Gamma(e,\lfloor n^\alpha\rfloor)|\\
    &=n^{\frac{1}{2}(\beta-\alpha)\log_qK}|B_\Gamma(e,\lfloor n^\alpha\rfloor)|\\
    &\ge n^{\frac{1}{2}(\beta-\alpha)\log_qK}|B_\Gamma(e,1)|,
\end{align*}
which by \eqref{eq:graph.poly.growth} implies that $K<q^{2d/(\beta-\alpha)}$. The lemma therefore holds with $K=q^{2d/(\beta-\alpha)}$.
\end{proof}

\begin{proof}[Proof of \cref{cor:trof}]
Provided $n$ is large enough in terms of $\lambda$ only, \cref{lem:poly.pigeon} implies that there exists $m\in\N$ with $\lfloor n^{\lambda/4}\rfloor\le m\le n^{\lambda/2}$ such that $\beta_\Gamma(5m)\le3^{8d/\lambda}\beta_\Gamma(m)$. Provided $n$ is large enough in terms of $d$ and $\lambda$, the desired result follows from \cref{thm:main}.
\end{proof}

Before we prove \cref{cor:trof.finite.cyc}, it will be convenient to record a slight variant of Breuillard and the second author's corresponding result in the special case of a Cayley graph. If $G$ is a finite group, $H\lhd G$, and $S$ is a symmetric generating set for $G$ containing the identity, then for brevity we write $\diam_S(G/H)$ to mean $\diam_{\pi(S)}(G/H)$, where $\pi:G\to G/H$ is the quotient homomorphism.

\begin{theorem}\label{thm:bt}
For every $\delta>0$ there exists $n_0=n_0(\delta)\ge1$ such that the following holds. Let $G$ be a finite group with a symmetric generating set $S$ containing $1$, and suppose that $\diam_S(G)\ge n_0$ and
\[
\diam_S(G)\ge\left(\frac{|G|}{|S|}\right)^\delta.
\]
Then there exist subgroups $U\lhd G'\lhd G$ such that $G'$ has index $O_\delta(1)$ in $G$ and $G'/U$ is cyclic, and a symmetric generating set $S'$ for $G'$ satisfying $S'\subset S^{O_\delta(1)}$ such that $d_{S'}\le d_S$ on $G'$, and such that $\diam_{S'}(G'/U)\gg_\delta\diam_S(G)$.
\end{theorem}
\begin{proof}
This is essentially \cite[Theorem 4.1 (3)]{bt}, although some of the desired properties of $S'$ are stated separately in \cite[Lemma 4.2]{bt}. The only property not stated explicitly in either \cite[Theorem 4.1 (3)]{bt} or \cite[Lemma 4.2]{bt} is that $d_{S'}\le d_S$ on $G'$, but this is given by the result \cite[Lemma 7.2.2]{hall} referenced in the proof of \cite[Lemma 4.2]{bt}.
\end{proof}

\begin{proof}[Proof of \cref{cor:trof.finite,cor:trof.finite.cyc}]
Writing $\gamma=\diam(\Gamma)$, the hypothesis \eqref{eq:trof.finite.hyp} translates as
\[
\beta_\Gamma(\gamma)\le\gamma^{1/\delta}\beta_\Gamma(1).
\]
Provided $\gamma$ is large enough in terms of $\delta$ only, \cref{cor:trof} therefore gives a normal subgroup $H_0\lhd G$ such that
\begin{enumerate}[label=(\alph*)]
\item\label{item:c.a} every fibre of the projection $\Gamma\to\Gamma/H_0$ has diameter at most $\gamma^{1/2}$;
\item\label{item:c.b} $G_{\Gamma/H_0}=G/H_0$;
\item\label{item:c.c} the set $S_1=\{g\in G_{\Gamma/H_0}:d_{\Gamma/H_0}(g(H_0(e)),H_0(e))\le1\}$ is a symmetric generating set for $G_{\Gamma/H_0}$; and
\item\label{item:c.d} every vertex stabiliser of the action of $G_{\Gamma/H_0}$ on $\Gamma/H_0$ has cardinality $O_\delta(1)$.
\end{enumerate}
We will use property \ref{item:c.b} implicitly throughout this proof in order to interchange $G_{\Gamma/H_0}$ and $G/H_0$. Setting $S_0=\{g\in G:d(g(e),e)\le1\}$, property \ref{item:c.a} and \cref{lem:cay-ab.gen.set} imply that
\begin{equation}\label{eq:trof.fin.H0}
H_0\subset S_0^{\lfloor\gamma^{1/2}\rfloor}.
\end{equation}
Write $\pi:G\to G/H_0$ and $\psi:\Gamma\to\Gamma/H_0$ for the quotient maps. \cref{lem:ball.preimage} combines with \eqref{eq:trof.fin.H0} to imply that $\psi^{-1}(B_{\Gamma/H_0}(H_0(e),\gamma-\lfloor\gamma^{1/2}\rfloor-1))\subset B_\Gamma(e,\gamma-1)\subsetneqq\Gamma$,
and hence that $\diam(\Gamma/H_0)\ge\gamma-\gamma^{1/2}$. \cref{lem:cay-ab.gen.set} implies that $\diam_{S_1}(G/H_0)=\diam(\Gamma/H_0)$, and so as long as $\gamma\ge4$ we conclude that
\begin{equation}\label{eq:diam.quot.large}
\diam_{S_1}(G/H_0)\ge\gamma-\gamma^{1/2}\ge\gamma^{1/2}.
\end{equation}
On the other hand, \cref{lem:cay-ab.gen.set,cor:growth.inherit} imply that
\[
\frac{|G|}{|S_0|}=\frac{|\Gamma|}{\beta_{\Gamma}(1)}.
\]
Since $|G/H_0|=|G|/|H_0|$ and $|S_1|\ge|S_0|/|H_0|$ by \cref{lem:image.of.S}, this combines with \eqref{eq:trof.finite.hyp} to imply that
\[
\frac{|G/H_0|}{|S_1|}\le\gamma^{1/\delta}.
\]
This in turn combines with \eqref{eq:diam.quot.large} to imply that
\begin{equation}\label{eq:trof.fin.red.cay}
\diam_{S_1}(G/H_0)\ge\left(\frac{|G/H_0|}{|S_1|}\right)^{\delta/2}.
\end{equation}

We first prove that the conclusions of \cref{cor:trof.finite} hold. Provided $\gamma$ is large enough in terms of $\delta$ and $\lambda$, applying \cite[Theorem 4.1 (2)]{bt} and \cref{lem:image.of.S} therefore implies that there exists $H\lhd G$ satisfying $H_0\subset H\subset S_0^{\lfloor\gamma^{\frac{1}{2}+\lambda/2}\rfloor}H_0$ such that $G/H$ has an abelian subgroup of rank and index at most $O_{\delta,\lambda}(1)$, yielding property \ref{item:c.iii.fin}. By \eqref{eq:trof.fin.H0}, provided $\gamma$ is large enough in terms of $\lambda$ we may conclude that
\begin{equation}\label{eq:trof.fin.H}
H\subset S_0^{\lfloor\gamma^{\frac{1}{2}+\frac{2}{3}\lambda}\rfloor},
\end{equation}
which in particular gives \ref{item:c.i.fin} by \cref{lem:cay-ab.gen.set}. \cref{lem:make.H.kernel} implies that we may also assume that \ref{item:c.ii.fin} holds. Property \ref{item:c.iv.fin} follows from \cref{lem:cay-ab.gen.set}, whilst \ref{item:c.v.fin} follows from \ref{item:c.d} and \cref{lem:image.of.S}. Finally, \eqref{eq:trof.fin.H} combines with \cref{lem:quotient.QI} to show that the quotient map $\Gamma\to\Gamma/H$ is a $(1,\gamma^{\frac{1}{2}+\frac{2}{3}\lambda})$-quasi-isometry, whilst \cref{lem:aut.QI} implies that there is a $(1,1)$-quasi-isometry $\Cay(G_{\Gamma/H},S)\to\Gamma/H$, so provided $\gamma$ is large enough in terms of $\lambda$, \ref{item:c.vi.fin} follows from \cref{lem:QI.elem}. This proves \cref{cor:trof.finite}.

We now prove that the conclusions of \cref{cor:trof.finite.cyc} hold. Provided $\gamma$ is large enough in terms of $\delta$, \eqref{eq:trof.fin.red.cay} means that we may apply \cref{thm:bt}. This gives a normal subgroup $G'$ of index at most $O_\delta(1)$ in $G$ (as required by conclusion \ref{item:fin.ind}), and a normal subgroup $U\lhd G'$ containing $H_0$ such that $G'/U$ is cyclic (as required by conclusion \ref{item:cyclic}). It also gives a symmetric generating set $S_1'$ for $\pi(G')$ containing the identity such that
\begin{equation}\label{eq:biLip}
d_{S_1}\ll_\delta d_{S_1'}\le d_{S_1}
\end{equation}
on $\pi(G')$ and such that $\diam_{S_1'}(G'/U)\gg_\delta\diam_{S_1}(G/H_0)$. Taking $S'$ to be the image of $S_1'$ in $G'/U$, if $\gamma$ is large enough then \eqref{eq:diam.quot.large} then implies that
$\diam_{S'}(G'/U)\gg_\delta\gamma$, giving conclusion \ref{item:diam}.
Conclusion \ref{item:small.stabs} follows from \cref{lem:image.of.S} and \ref{item:c.d}. Finally, to prove \ref{item:biLip}, note that for $g_1,g_2\in G'$ and $\lambda>0$ we have
\begin{align*}
d_{G'(e)/U}(g_1U(e),g_2U(e))\ll_\delta\lambda
   &\iff\min_{u\in U}d_{\Gamma/H_0}(g_2^{-1}g_1uH_0(e),H_0(e))\ll_\delta\lambda&\text{(by \cref{lem:image.of.S})}\\
   &\iff\min_{u\in U}d_{S_1}(g_2^{-1}g_1uH_0,H_0)\ll_\delta\lambda&\text{(by \cref{lem:cay-ab.gen.set})}\\
   &\iff\min_{u\in U}d_{S_1'}(g_2^{-1}g_1uH_0,H_0)\ll_\delta\lambda&\text{(by \eqref{eq:biLip})}\\
   &\iff d_{S'}(g_1U,g_2U)\ll_\delta\lambda.
\end{align*}
\end{proof}

\section{Growth of vertex-transitive graphs with one ball of polynomial size}
In this section we prove \cref{cor:tao.growth,cor:persist.abs,cor:mod.growth}
\begin{prop}\label{prop:red.to.cayley}
In the setting of \cref{thm:main} we have
\[
\frac{|(G_{\Gamma/H})_{H(e)}|}{|H(e)|}\beta_\Gamma(m)\le|S^m|\ll_K\frac{|(G_{\Gamma/H})_{H(e)}|}{|H(e)|}\beta_\Gamma(m+O_K(n))
\]
for every $m\in\N$. In particular, $|S^m|\ll_K\beta_\Gamma(m+O_K(n))$ and
\[
\frac{|S^m|}{|S^{m'}|}\ll_K\frac{\beta_\Gamma(m+O_K(n))}{\beta_\Gamma(m')}
\]
for every $m,m'\in\N$.
\end{prop}
\begin{proof}
Let $m\in\N$. \cref{cor:growth.inherit} implies that
\begin{equation}\label{eq:ball.x.stab}
|S^m|=|(G_{\Gamma/H})_{H(e)}|\,\beta_{\Gamma/H}(m)
\end{equation}
for every $m\in\N$. \cref{lem:fibres.isom} implies that every fibre of the projection $\psi:\Gamma\to\Gamma/H$ has size $|H(e)|$. Since $\psi(B_\Gamma(e,m))\subset B_{\Gamma/H}(H(e),m)$, this implies that
\begin{equation}\label{eq:ball.lb}
\beta_{\Gamma/H}(m)\ge\frac{\beta_\Gamma(m)}{|H(e)|}.
\end{equation}
The fact that the fibres of $\psi$ all have size $|H(e)|$ also implies that
\[
|\psi^{-1}(B_{\Gamma/H}(H(e),m))|=|H(e)|\,\beta_{\Gamma/H}(m),
\]
 and so \cref{lem:cay-ab.gen.set,lem:ball.preimage} imply that
\begin{equation}\label{eq:ball.ub}
\beta_{\Gamma/H}(m)\le\frac{\beta_\Gamma(m+O_K(n))}{|H(e)|}.
\end{equation}
The desired result follows from combining \eqref{eq:ball.x.stab}, \eqref{eq:ball.lb} and \eqref{eq:ball.ub}.
\end{proof}

The proofs of \cref{cor:tao.growth,cor:persist.abs} begin in roughly the same way, so to avoid repitition we isolate the common part of the arguments as follows.

\begin{lemma}\label{lem:poly.reduc.to.cay}
Given $d>0$ there exists $n_0=n_0(d)$ such that if $\Gamma$ is a connected locally finite vertex-transitive graph satisfying $\beta_\Gamma(n)\le n^{d+1}\beta_\Gamma(1)$ for some $n\ge n_0$ then there exists a group with a finite symmetric generating set $S$ containing the identity such that
\begin{equation}\label{eq:persist.cay}
|S^r|\ll_\fd\beta_\Gamma(r+\textstyle\lfloor\frac{n}{2}\rfloor)
\end{equation}
and
\begin{equation}\label{eq:persist.cay.2}
\frac{|S^{r'}|}{|S^r|}\ll_\fd\frac{\beta_\Gamma(r'+\textstyle\lfloor\frac{n}{2}\rfloor)}{\beta_\Gamma(r)}
\end{equation}
for every $r,r'\in\N$.
\end{lemma}
\begin{proof}
Provided $n\ge 3^8$, \cref{lem:poly.pigeon} implies that there exists $n'\in\N$ with $\lfloor n^{1/2}\rfloor\le n'\le n^{3/4}$ such that $\beta_\Gamma(3n')\ll_d\beta_\Gamma(n')$. Provided $n$ is large enough in terms of $d$ only, applying \cref{thm:main} and \cref{prop:red.to.cayley} to the whole of $\Aut(\Gamma)$ then gives $S$ with the desired properties.
\end{proof}

In proving \cref{cor:tao.growth} we use the following elementary property of continuous piecewise-monomial functions.
\begin{lemma}\label{lem:p/w.mono}
Let $c>1$ and $d\ge0$. Let $f:[1,\infty)\to[1,\infty)$ be a non-decreasing continuous piecewise-monomial function with finitely many pieces, each of which has degree at most $d$. Then $f(cx)\le c^df(x)$ for every $x\ge1$.
\end{lemma}
\begin{proof}
This statement is completely trivial if $x$ and $cx$ lie in the same monomial piece of $f$, and if there exists $c'\in(1,c)$ such that $c'x$ is the boundary of two monomial pieces then we have $f(cx)\le(c/c')^df(c'x)\le c^df(x)$ by induction on the number of such boundaries lying in $(x,cx)$.
\end{proof}

\begin{proof}[Proof of \cref{cor:tao.growth}]
Suppose that $\beta_\Gamma(n)\le n^d\beta_\Gamma(1)$, and let $S$ be the set obtained by applying \cref{lem:poly.reduc.to.cay}. Taking $r=1$ and $r'=\lceil n/2\rceil$ in \eqref{eq:persist.cay.2}, it follows that $|S^{\lceil n/2\rceil}|\le n^d|S|\le\textstyle\lceil\frac{n}{2}\rceil^{2d}|S|$,
the last inequality being valid as long as $n\ge3$. Provided $n$ is large enough in terms of $d$ only, \cite[Theorem 1.9]{tao.growth} therefore gives a non-decreasing continuous piecewise-monomial function $f:[1,\infty)\to[1,\infty)$ with $f(1)=1$ and at most $O_d(1)$ distinct pieces, each of degree a non-negative integer at most $O_d(1)$, such that
\begin{equation}\label{eq:tao.growth.conc}
\frac{|S^{m\lceil n/2\rceil}|}{|S^{\lceil n/2\rceil}|}\asymp_df(m)
\end{equation}
for every $m\in\N$.

Now fix $m\in\N$. Setting $r=2m\lceil n/2\rceil$ and $r'=\lceil\frac{n}{2}\rceil$ in \eqref{eq:persist.cay.2} implies that
\[
\frac{\beta_\Gamma(mn)}{\beta_\Gamma(n)}\le\frac{\beta_\Gamma(2m\lceil\frac{n}{2}\rceil)}{\beta_\Gamma(n)}
      \ll_d\frac{|S^{2m\lceil n/2\rceil}|}{|S^{\lceil n/2\rceil}|},
\]
and so \eqref{eq:tao.growth.conc} implies that $\beta_\Gamma(mn)\ll_d f(2m)\beta_\Gamma(n)$ and then \cref{lem:p/w.mono} gives $\beta_\Gamma(mn)\ll_d f(m)\beta_\Gamma(n)$,
as required. On the other hand, setting $r=n$ and $r'=mn-\lfloor\frac{n}{2}\rfloor$ in \eqref{eq:persist.cay.2} implies that
\[
\frac{\beta_\Gamma(mn)}{\beta_\Gamma(n)}\gg_d\frac{|S^{mn-\lfloor n/2\rfloor}|}{|S^n|}
   \ge\frac{|S^{m\lceil n/2\rceil}|}{|S^{2\lceil n/2\rceil}|}
   =\frac{|S^{m\lceil n/2\rceil}|}{|S^{\lceil n/2\rceil}|}\frac{|S^{\lceil n/2\rceil}|}{|S^{2\lceil n/2\rceil}|}.
\]
It therefore follows from \eqref{eq:tao.growth.conc} that $\beta_\Gamma(mn)\gg_d f(m)\beta_\Gamma(n)/f(2)$, and so \cref{lem:p/w.mono} combines with the fact that $f(1)=1$ to imply that $\beta_\Gamma(mn)\gg_d f(m)\beta_\Gamma(n)$,
as required.
\end{proof}

Before proving \cref{cor:persist.abs} it will be convenient to recall from our previous paper the corresponding result for Cayley graphs.
\begin{theorem}\label{thm:persist.abs.cay}
Given $d\in\N$ there exists $C=C_d>0$ such that if $S$ is a finite symmetric generating set for a group $G$ such that $1\in S$ and $|S^n|\le n^{d+1}/C$ for some $n\ge C$ then for every $m\ge n$ we have $|S^m|\ll_d(m/n)^d|S^n|$.
\end{theorem}
\begin{proof}
This follows directly by taking $C$ to be the constant $N_d$ coming from \cite[Theorem 1.11]{tt}, and then applying that theorem with $D=d$ and $M=n/N_d$.
\end{proof}
\begin{remark*}We leave it to the reader to verify that, conversely, Theorem \ref{thm:persist.abs.cay} implies \cite[Theorem 1.11]{tt}.
\end{remark*}

\begin{proof}[Proof of \cref{cor:persist.abs}]
The fact that $\beta_\Gamma(n)\le n$ implies $\Gamma=B_\Gamma(x,n)$ simply follows from the fact that $\beta_\Gamma(0)=1$ and $\beta_\Gamma(m+1)>\beta_\Gamma(m)$ for every $m<\diam(\Gamma)$. To prove the rest of the corollary, assuming \eqref{eq:persist.abs.hyp} holds, let $S$ be the set obtained by applying \cref{lem:poly.reduc.to.cay}. Applying \eqref{eq:persist.cay} with $r=\lceil n/2\rceil$ then implies that
\begin{equation}\label{eq:benj.hyp.cay}
|S^{\lceil n/2\rceil}|\ll_d\beta_\Gamma(n)\ll_d\frac{(\textstyle\lceil\frac{n}{2}\rceil)^{d+1}}{C}.
\end{equation}
Provided $C$ is large enough in terms of $d$ only, \cref{thm:persist.abs.cay} therefore implies that $|S^m|\ll_d(m/n)^d|S^{\lceil n/2\rceil}|$ for every $m\ge\lceil n/2\rceil$, and so \cref{cor:persist.abs} follows from applying \eqref{eq:persist.cay.2} with $r=m$ and $r'=\lceil n/2\rceil$.
\end{proof}

\begin{proof}[Proof of \cref{cor:mod.growth}]
Taking $n=1$ in the definition of moderate growth immediately gives \eqref{eq:mod.growth.conc}. For the converse we use a similar argument to \cite{bt}. Write $\gamma=\diam(\Gamma)$, and write $d=\delta^{-1}$. Let $n\in\N$. First, suppose that $\gamma^{1/2}\le n\le\gamma$, and note then that \eqref{eq:mod.growth.hyp} implies that $\beta_\Gamma(n)\le n^{2d}\beta_\Gamma(1)$. Provided $\gamma$ is large enough in terms of $\delta$, \cref{cor:tao.growth,lem:p/w.mono} then imply that $\beta_\Gamma(mn)\ll_\delta m^{O_\delta(1)}\beta_\Gamma(n)$ for every $m\ge n$, and hence in particular that $|\Gamma|=\beta_\Gamma(\gamma)\ll_\delta \left(\frac{\gamma}{n}\right)^{O_\delta(1)}\beta_\Gamma(n)$.
This implies that $\Gamma$ satisfies the definition of $(O_\delta(1),O_\delta(1))$-moderate growth for this value of $n$. If $n\le\gamma^{1/2}$, on the other hand, then \eqref{eq:mod.growth.hyp} implies that
\[
\frac{\beta_\Gamma(n)}{|\Gamma|}\ge\frac{\beta_\Gamma(1)}{|\Gamma|}\ge\frac{1}{\gamma^d}\ge\left(\frac{n}{\gamma}\right)^{2d},
\]
and so $\Gamma$ satisfies the definition of $(1,2d)$-moderate growth for this value of $n$.
\end{proof}


\begin{thebibliography}{10}
\bibitem{bk}
I. Benjamini and G. Kozma. A resistance bound via an isoperimetric inequality, \textit{Combinatorica} \textbf{25}(6) (2005), 645--650.
\bibitem{bft}
I. Benjamini, H. Finucane and R. Tessera. On the scaling limit of finite vertex transitive graphs with large diameter.  \textit{Combinatorica} \textbf{36} (2016), 1--41.
\bibitem{bgt}
E. Breuillard, B. J. Green and T. C. Tao. The structure of approximate groups, \textit{Publ. Math. IHES.} \textbf{116}(1) (2012), 115--221.
\bibitem{bt}
E. Breuillard and M. C. H. Tointon. Nilprogressions and groups with moderate growth, \textit{Adv. Math.} \textbf{289} (2016), 1008--1055.
\bibitem{Ca}
P. K. Carolino. The Structure of Locally Compact Approximate Groups. PhD thesis: \url{https://escholarship.org/uc/item/8388n9jk}.
\bibitem{dsc}
P. Diaconis and L. Saloff-Coste. Moderate growth and random walk on finite groups, \textit{Geom. Funct. Anal.} \textbf{4}(1) (1994), 1--36.
\bibitem{dl}
R. Diestel and I. Leader. A conjecture concerning a limit of non-Cayley graphs, \textit{J. Algebraic Combin.} \textbf{14} (2001), 17--25.
\bibitem{efw}
A. Eskin, D. Fisher and K. Whyte. Quasi-isometries and rigidity of solvable groups, \textit{Pure Appl. Math. Q.} \textbf{3} (2007), Special Issue: In honor of Grigory Margulis. Part 1, 927--947.
\bibitem{gromov}
M. Gromov. Groups of polynomial growth and expanding maps, \textit{Publ. Math. IHES} \textbf{53} (1981), 53--73.
\bibitem{hall}
M. Hall. \textit{The theory of groups}, Amer. Math. Soc./Chelsea, Providence, RI (1999).
\bibitem{her-pym}
J. Hermon and R. Pymar. The exclusion process mixes (almost) faster than independent particles, arXiv:1808.10846v2.
\bibitem{hew-ross}
E. Hewitt and K. A. Ross. \textit{Abstract Harmonic Analysis I (2nd ed.)}, Springer-Verlag, Berlin (1979).
\bibitem{kleiner}
B. Kleiner. A new proof of Gromov's theorem on groups of polynomial growth, \textit{Jour. of the AMS} \textbf{23}(3) (2010), 815--829.
\bibitem{lev-per}
D. A. Levin and Y. Peres, with contributions by E. L. Wilmer. \textit{Markov Chains and Mixing Times (2nd ed.)}, American Mathematical Society, Providence, RI (2017).
\bibitem{losert}
V. Losert. On the structure of groups with polynomial growth, \textit{Math. Z.} \textbf{195} (1987), 109--117.
\bibitem{ozawa}
N. Ozawa. A functional analysis proof of Gromov's polynomial growth theorem, \textit{Ann. Sci. Éc. Norm. Supér. (4)} \textbf{51}(3) (2018), 549--556.
\bibitem{shalom-tao}
Y. Shalom and T. C. Tao. A finitary version of Gromov's polynomial growth theorem, \textit{Geom. Funct. Anal.} \textbf{20}(6) (2010), 1502--1547.
\bibitem{tao.product.set}
T. C. Tao. Product set estimates for non-commutative groups, \textit{Combinatorica} \textbf{28}(5) (2008), 547--594.
\bibitem{tao.growth}
T. C. Tao. Inverse theorems for sets and measures of polynomial growth, \textit{Q. J. Math.} \textbf{68}(1) (2017), 13--57.
\bibitem{tt}
R. Tessera and M. C. H. Tointon. Properness of nilprogressions and the persistence of polynomial growth of given degree, \textit{Discrete Anal.} 2018:17, 38 pp.
\bibitem{ttBK}
R. Tessera and M. C. H. Tointon. Sharp relations between volume growth, isoperimetry and resistance in vertex-transitive graphs, arXiv:2001.01467.
\bibitem{ttLie}
R. Tessera and M. C. H. Tointon. Lie group approximations of balls with polynomial volume in vertex-transitive graphs, in preparation.
\bibitem{nilp.frei}
M. C. H. Tointon. Freiman's theorem in an arbitrary nilpotent group, \textit{Proc. London Math. Soc.} (3) \textbf{109} (2014), 318--352.
\bibitem{book}
M. C. H. Tointon. \textit{Introduction to approximate groups}, London Mathematical Society Student Texts \textbf{94}, Cambridge University Press, Cambridge (2020).
\bibitem{trof}
V. I. Trofimov. Graphs with polynomial growth, \textit{Math. USSR-Sb.} \textbf{51} (1985) 405--417.
\bibitem{woess} W. Woess. Topological groups and infinite graphs,  \textit{Discrete Math.} \textbf{95} (1991), 373--384. 
\end{thebibliography}
\end{document}